\documentclass[12pt, leqno]{amsart}

\usepackage{amssymb, hyperref}

\usepackage{fullpage}

\usepackage[all]{xy}

\theoremstyle{plain}

\newtheorem{theorem}{Theorem}[section]
\newtheorem{lemma}[theorem]{Lemma}
\newtheorem{proposition}[theorem]{Proposition}

\theoremstyle{definition}

\newtheorem{remark}[theorem]{Remark}
\newtheorem{remarks}[theorem]{Remarks}
\newtheorem{example}[theorem]{Example}

\newcommand\bC{{\mathbb C}}

\newcommand\bG{{\mathbb G}}

\newcommand\bP{{\mathbb P}}
\newcommand\bQ{{\mathbb Q}}

\newcommand\bS{{\mathbb S}}

\newcommand\bX{{\mathbb X}}

\newcommand\cK{{\mathcal K}}
\newcommand\cL{{\mathcal L}}

\newcommand\cO{{\mathcal O}}

\newcommand\cU{{\mathcal U}}

\newcommand\fb{\mathfrak{b}}
\newcommand\fg{\mathfrak{g}}

\newcommand\fl{\mathfrak{l}}

\newcommand\fp{\mathfrak{p}}

\newcommand\hf{\widehat{f}}
\newcommand\hx{\widehat{x}}
\newcommand\hy{\widehat{y}}
\newcommand\hw{\widehat{w}}
\newcommand\hC{\widehat{C}}
\newcommand\hX{\widehat{X}}
\newcommand\hY{\widehat{Y}}
\newcommand\hpi{\widehat{\pi}}

\newcommand\rA{{\rm A}}
\newcommand\rB{{\rm B}}
\newcommand\rC{{\rm C}}
\newcommand\rD{{\rm D}}
\newcommand\rE{{\rm E}}

\newcommand\tf{{\widetilde f}}
\newcommand\tw{{\widetilde w}}
\newcommand\tx{{\widetilde x}}

\newcommand\tpi{{\widetilde \pi}}

\newcommand\tC{{\widetilde C}}

\newcommand\tX{{\widetilde X}}

\newcommand\bir{{\rm bir}}

\newcommand\charac{{\rm char}}

\renewcommand\deg{{\rm deg}}

\newcommand\emfr{{\rm emfr}}

\newcommand\free{{\rm fr}}

\newcommand\he{{\rm ht}}
\newcommand\id{{\rm id}}

\newcommand\n{{\rm n}}

\newcommand\pr{{\rm pr}}

\newcommand\sm{{\rm sm}}

\newcommand\Aut{{\rm Aut}}

\newcommand\Chow{{\rm Chow}}

\newcommand\Hom{{\rm Hom}}

\newcommand\Lie{{\rm Lie}}

\newcommand\Peaks{{\rm Peaks}}
\newcommand\Pic{{\rm Pic}}

\newcommand\RatCurves{{\rm RatCurves}}

\newcommand\SL{{\rm SL}}

\newcommand\Supp{{\rm Supp}}

\newcommand\Univ{{\rm Univ}}

\numberwithin{equation}{section}

\title{Minimal rational curves on generalized Bott-Samelson varieties}

\author{Michel Brion and S.~Senthamarai Kannan}

\date{}

\begin{document}

\begin{abstract}
We investigate families of minimal rational curves on Schubert 
varieties, their Bott-Samelson desingularizations, and their generalizations
constructed by Nicolas Perrin in the minuscule case. In particular, we describe 
the minimal families on small resolutions of minuscule Schubert varieties.
\end{abstract}

\maketitle

\section{Introduction}
\label{sec:int}

Lines in flag varieties have been extensively investigated. In particular,
for a homogeneous space $X = G/P$ where $G$ is a semi-simple
algebraic group and $P$ a maximal parabolic subgroup, the lines in $X$
passing through the base point $x$ form a smooth projective variety
$\cL_x$ on which $P$ acts with one or two orbits. When $P$ is
associated to a long simple root, $\cL_x$ is the $P$-orbit of the
Schubert line in $X$; moreover, the homogeneous space $\cL_x$ 
is minuscule (see \cite{CC, Strickland, LM} for these results).

The variety $\cL_x$ features prominently in work of Hwang and Mok 
establishing rigidity properties of $X$ (see e.g.~\cite{HwM}).
Its analogue for a smooth Schubert variety $Y$ of $X$ is an
important ingredient in the study of the deformations of $Y$ within 
$X$, by Hong et al.~(see \cite{HoM, Hong, HoKw}). But little seems 
to be known about lines in possibly singular Schubert varieties. 
The latter admit natural resolutions of singularities, 
the Bott-Samelson varieties and their generalizations introduced 
by Sankaran and Vanchinathan (see \cite{SV94, SV95}), and by 
Perrin in \cite{Pe07}. For these generalized Bott-Samelson 
resolutions, the notion of lines (which depends on a projective 
embedding) may be replaced with the intrinsic notion of minimal 
rational curves. In loose terms, a family of rational curves 
on a projective variety $X$ is minimal if the subfamily of curves 
through a general point $x \in X$ is non-empty and proper. 
The minimal families and the associated varieties of 
minimal rational tangents (consisting, in loose terms again, 
of the tangent directions at $x$ of the curves through that point) 
play an important r\^ole in the geometry of $X$, 
see e.g.~\cite{Hwang14}.

In this paper, we make the first steps in the investigation of lines 
in Schubert varieties, and minimal families in their generalized 
Bott-Samelson resolutions. Given a Schubert variety 
$Y$ in $X = G/P$ with $P$ maximal, it is easy to show that 
$Y$ is covered by translates of the Schubert line (see Lemma 
\ref{lem:covered} for details). If in addition $P$ is associated 
to a long simple root, then the Chow variety of lines through 
a general point $y \in Y$ is a union of Schubert varieties in 
$\cL_y$ (Proposition \ref{prop:linesbis}); it may be reducible  
(Example \ref{ex:linesbis}). 

For Bott-Samelson resolutions, the minimal families turn out 
to be very restricted: there are only finitely many minimal rational 
curves through a general point (see Theorem \ref{thm:bott} 
for details). This may be explained by the fact that Bott-Samelson 
resolutions are big (i.e., some fibers have a big dimension) and not 
canonical (indeed, the action of the connected automorphism group 
of the Schubert variety need not lift to the resolution; see 
\cite[\S 7]{CKP} for explicit examples).

By contrast, generalized Bott-Samelson resolutions include 
the small resolutions of minuscule Schubert varieties constructed 
by Zelevinsky (see \cite{Zelevinsky}), Sankaran and Vachinathan,
and Perrin in full generality (see \cite[Cor.~7.9]{Pe07}).
These resolutions are obtained as towers of locally trivial fibrations, 
with fibers being minuscule homogeneous spaces. We describe 
their minimal families in terms of lines in these homogeneous spaces 
(Theorem \ref{thm:final}). This relies on a structure result for minimal 
families on generalized Bott-Samelson resolutions (Proposition 
\ref{prop:isom}), and on two combinatorial properties of these resolutions
(Propositions \ref{prop:root} and \ref{prop:weyl}). Both properties 
were first proved in the companion article \cite{BK} via case-by-case
arguments using reduced decompositions in Weyl groups.
Then Perrin came up with uniform proofs based on the combinatorics 
of minuscule quivers developed in his papers \cite{Pe05,Pe07,Pe09}. 
Subsequently, we obtained somewhat shorter uniform proofs, which 
are presented here.

Our results are obtained over an algebraically closed field 
of arbitrary characteristic, whereas the setting of most earlier works 
is complex geometry. In particular, the theory of minimal rational 
curves seems to have been exclusively developed over $\bC$ so far. 
Thus, we only rely on foundational material from \cite[Chap.~II]{Kollar} 
(see also \cite[Chap.~II]{Debarre}). The facts that we need are 
gathered in \S \S \ref{subsec:spaces} and \ref{subsec:families}.

\S \ref{subsec:almost} contains auxiliary results
on almost homogeneous varieties, i.e., those on which
an algebraic group acts with an open dense orbit; this class
includes Schubert varieties, their (generalized) Bott-Samelson
resolutions, and some naturally associated varieties. In this 
special setting, we obtain analogues of important general 
results on the existence and properties of free rational curves, 
which hold over $\bC$ but generally fail in positive characteristics
(see Remark \ref{rem:nofree}). For this, we develop methods 
from \cite[\S 2]{BF}.

In \S \S \ref{subsec:flag} and \ref{subsec:schubert}, we set up
notation and recall basic facts on flag varieties and their Schubert 
varieties. The minimal rational curves on the former are described 
in \S \ref{subsec:lines}, whereas \S \ref{subsec:linesbis} explores
the families of lines on the latter. An essential r\^ole is played
by the curves which are stable by a maximal torus $T$ of $G$.
Our results are most complete for minuscule homogeneous spaces; 
these may be characterized by the condition that every $T$-stable 
curve is a line.

The minimal rational curves on Bott-Samelson desingularizations 
are considered in \S \ref{subsec:bott}. In \S \ref{subsec:perrin},
we survey the construction of their generalizations, due to Perrin 
in \cite[\S 5]{Pe07}. The structure of their minimal families is 
investigated in \S \ref{subsec:structure}; again, $T$-stable curves form 
a key ingredient in all these developments. We illustrate our results 
on the simplest example of a singular Schubert variety: a quadratic 
cone of dimension $3$ (Examples \ref{ex:linesbis} and \ref{ex:final}). 
Further examples, also involving exceptional groups, can be found 
in \cite{BK}. 

The final Section \ref{sec:comb} is devoted to combinatorial 
properties of generalized Bott-Samelson desingularizations. 
In \S \ref{subsec:root}, we obtain an inequality involving certain 
positive roots; an equality of Weyl groups of isotropy groups
is proved in \S \ref{subsec:weyl}.

Our approach raises many open questions: for example,
to parameterize the families of lines in a Schubert variety in
combinatorial terms, and to characterize those families that
contain lines consisting of smooth points. Also, it would be 
interesting to extend our results to the setting of cominuscule 
homogeneous varieties, or to the horizontal Schubert varieties
introduced in \cite{KR}.

\bigskip

\noindent
{\bf Acknowledgments.} 
We thank Jaehyun Hong for helpful email exchanges, and the referee 
for a careful reading and valuable corrections. Special thanks 
are due to Nicolas Perrin for his insightful comments and suggestions; 
as mentioned above, he first obtained uniform proofs of two key 
combinatorial results, and the proofs presented here are partly based 
on his arguments.

The second-named author would like to thank the Institut Fourier 
for the hospitality during his stay. He also thank the Infosys Foundation
for the partial financial support.

\medskip

\noindent
{\bf Notation and conventions.}
The ground field $k$ is algebraically closed, of characteristic 
$p \geq 0$. By a \emph{scheme}, we mean a separated $k$-scheme
$S$; \emph{points} of $S$ are $k$-rational points unless otherwise
stated. A \emph{variety} is an integral scheme of finite type over $k$. 
A \emph{curve} is a variety of dimension $1$.

An \emph{algebraic group} is a group scheme of finite type. 
Given an algebraic group $G$, a subgroup scheme $H$
and a scheme $Y$ equipped with an action of $H$, we denote
by $G \times^H Y$ the quotient of $G \times Y$ by the $H$-action
via $h \cdot (g,y) := (gh^{-1},h y)$, if this quotient exists as a scheme.
We then have a cartesian square 
\[ \xymatrix{
G \times Y \ar[r]^-{\pr_1} \ar[d]_{q} & G \ar[d]^{r} \\
G \times^H Y \ar[r]^-{f} & G/H, \\
} \]
where $\pr_1$ denotes the first projection and $q,r$ are
principal $H$-bundles. As a consequence, $f$ is faithfully flat. 
Moreover, the $G$-action on $G \times Y$ via left multiplication 
on $G$ descends to a unique action on $G \times^H Y$, and 
$f$ is $G$-equivariant. We may view $f$ as a homogeneous 
fibration with fiber $Y$.

By \cite[Prop.~7.1]{MFK}, the associated fiber bundle 
$G \times^H Y$ exists if $Y$ admits an ample $H$-linearized 
line bundle. We will freely use the following observation:
if $X$ is a scheme equipped with an action of $G$ and 
an equivariant morphism $\pi : X \to G/H$ with fiber $Y$
at the base point of $G/H$, then there is a unique 
$G$-equivariant isomorphism $X \simeq G \times^H Y$ 
identifying $\pi$ with $f$.

\section{Rational curves on almost homogeneous varieties}
\label{sec:rational}

\subsection{Spaces of rational curves}
\label{subsec:spaces}

Let $X$ be a projective variety. The scheme of morphisms 
$\Hom(\bP^1,X)$ is equipped with an action of $\Aut(\bP^1)$ 
that stabilizes the open subscheme $\Hom_{\bir}(\bP^1,X)$ 
consisting of morphisms which are birational to their image. 
Moreover, this action lifts uniquely to an action on the normalization 
\[ \eta : \Hom_{\bir}^{\n}(\bP^1,X) \longrightarrow \Hom_{\bir}(\bP^1,X) \]
By \cite[II.2.15]{Kollar}, there is a natural commutative diagram 
of normal schemes
\begin{equation}\label{eqn:hom}
\xymatrix{
\bP^1 \times \Hom_{\bir}^{\n}(\bP^1,X) \ar[r]^-{\lambda} \ar[d]_{\pr_2} & 
\Univ(X) \ar[d]^{\rho} \\
\Hom_{\bir}^{\n}(\bP^1,X) \ar[r]^{\kappa} & \RatCurves(X), \\
}
\end{equation}
where the horizontal arrows are principal $\Aut(\bP^1)$-bundles. 
As a consequence, the above diagram is cartesian; moreover,
$\rho$ is a $\bP^1$-bundle. 

In view of \cite[II.2.11]{Kollar}, there is another natural 
commutative diagram
\begin{equation}\label{eqn:chow}
\xymatrix{
\Univ(X) \ar[r]^-{\delta} \ar[d]_{\rho} & 
\Chow(X) \times X \ar[d]^{\pr_1} \\
\RatCurves(X) \ar[r]^{\gamma} & \Chow(X), \\
}
\end{equation}
where $\Chow(X)$ denotes the Chow scheme. Moreover,  
$\gamma$ is finite over its image, which is the locally 
closed subscheme of $\Chow(X)$ parameterizing irreducible, 
geometrically rational $1$-cycles; also, $\delta$ is finite 
over its image, which is the universal Chow family over 
the above subscheme. By composing $\delta$ with the second 
projection, we obtain a morphism 
\begin{equation}\label{eqn:mu}
\mu : \Univ(X) \longrightarrow X 
\end{equation}
such that the morphism 
$\rho \times \mu : \Univ(X) \to \RatCurves(X) \times X$
is finite. 

The morphism (\ref{eqn:mu}) can be constructed alternatively
as follows: by composing the evaluation map
\begin{equation}\label{eqn:ev} 
\bP^1 \times \Hom_{\bir}(\bP^1,X) \longrightarrow X,
\quad (t,f) \longmapsto f(t) 
\end{equation}
with the map 
$\id \times \eta : \bP^1 \times \Hom_{\bir}^{\n}(\bP^1,X) 
\longrightarrow \bP^1 \times \Hom_{\bir}(\bP^1,X)$, 
we obtain a morphism
\[ \varepsilon : \bP^1 \times \Hom_{\bir}^{\n}(\bP^1,X) 
\longrightarrow X. \]
One may check that $\varepsilon$ is the composition of 
the quotient map
\[ \lambda : \bP^1 \times \Hom_{\bir}^{\n}(\bP^1,X) \to \Univ(X) \] 
with $\mu$. In particular, for any $x \in X$, we have a principal 
$\Aut(\bP^1)$-bundle $\varepsilon^{-1}(x) \to \mu^{-1}(x)$
between (scheme-theoretic) fibers. The first projection
$\pr_1 : \varepsilon^{-1}(x) \to \bP^1$ is $\Aut(\bP^1)$-equivariant,
and $\bP^1$ may be identified with the homogeneous space
$\Aut(\bP^1)/\Aut(\bP^1,0)$. This identifies 
$\varepsilon^{-1}(x)$ with the associated fiber bundle
$\Aut(\bP^1)\times^{\Aut(\bP^1,0)} \pr_1^{-1}(0)$.
Also, note that 
\[ \pr_1^{-1}(0) \simeq \eta^{-1}(\Hom_{\bir}(\bP^1,X; 0 \mapsto x)) \]
equivariantly for $\Aut(\bP^1,0)$, where 
$\Hom_{\bir}(\bP^1,X; 0 \mapsto x)$ denotes the closed subscheme
of $\Hom_{\bir}(\bP^1,X)$ consisting of those morphisms $f$
such that $f(0) = x$. Putting all of this together, we obtain a principal
$\Aut(\bP^1,0)$-bundle
\begin{equation}\label{eqn:bundle} 
\eta^{-1}(\Hom_{\bir}(\bP^1,X; 0 \mapsto x)) \to \mu^{-1}(x). 
\end{equation} 
We may view $\mu^{-1}(x)$ as the space of rational curves on
$X$ through $x$.

Another space of rational curves on $X$ through $x$ is constructed
in \cite[II.2.16]{Kollar}. More specifically, there is a natural 
commutative diagram of normal schemes
\[ \xymatrix{
\bP^1 \times \Hom_{\bir}^{\n}(\bP^1,X; 0 \mapsto x) 
\ar[r] \ar[d]_{\pr_2} & \Univ(x,X) \ar[d] \\
\Hom_{\bir}^{\n}(\bP^1,X; 0 \mapsto x) \ar[r] & \RatCurves(x,X), \\
} \] 
where the horizontal arrows are principal $\Aut(\bP^1,0)$-bundles
and the vertical arrows are $\bP^1$-bundles. (Here 
$\Hom_{\bir}^{\n}(\bP^1,X; 0 \mapsto x)$ denotes the normalization of
$\Hom_{\bir}(\bP^1,X; 0 \mapsto x)$). In general, the relation between 
$\mu^{-1}(x)$ and $\RatCurves(x,X)$ is not clear to us. But we will see
that both share a common smooth open subscheme, consisting of 
the free curves through $x$; these may be defined as follows.

Let $f \in \Hom_{\bir}(\bP^1,X)$, and $C$ its image. We say that 
$f$ is \emph{free} if $C$ is contained in the smooth locus $X_{\sm}$ 
and the vector bundle $f^*(T_{X_{\sm}})$ is generated by its global
sections, where $T_{X_{\sm}}$ denotes the tangent bundle of
$X_{\sm}$. 

Every free morphism $f$ satisfies $H^1(\bP^1,f^*(T_{X_{\sm}})) = 0$. 
Thus, the above notion of freeness coincides with that of 
\cite[II.3.1]{Kollar} when $X$ is smooth. By [loc.~cit.],
the free morphisms form a smooth open subscheme 
$\Hom_{\free}(\bP^1,X)$ of $\Hom(\bP^1,X)$; its dimension 
at the point $C$ is $- K_{X_{\sm}} \cdot C +\dim(X)$, 
where $K_{X_{\sm}}$ stands for the canonical class of 
the smooth locus. We denote by $\RatCurves_{\free}(X)$ 
the corresponding smooth open subscheme of $\RatCurves(X)$.

We also have $H^1(\bP^1,f^*(T_{X_{\sm}})(-1)) = 0$ whenever 
$f$ is free. As a consequence, the free morphisms that send
$0$ to $x$ form a smooth open subscheme 
$\Hom_{\free}(\bP^1,X; 0 \to x)$ of $\Hom(\bP^1,X; 0 \to x)$,
with dimension at $C$ being $- K_{X_{\sm}} \cdot C$ (see 
\cite[II.1.7, II.3.2]{Kollar} for these results).

Thus, any free morphism $f$ yields a smooth point $C$ of
$\RatCurves(X)$. Since the evaluation map (\ref{eqn:ev})
is smooth along $\bP^1 \times f$ (see \cite[II.3.5.4]{Kollar}),
$\mu^{-1}(x)$ is smooth at $C$ as well. As a consequence,
$\mu^{-1}(x)$ and $\RatCurves(x,X)$ share a common
smooth open subscheme $\RatCurves_{\free}(x,X)$, the quotient of  
$\Hom_{\free}(\bP^1,X; 0 \to x)$ by $\Aut(\bP^1,0)$.
The dimension of $\RatCurves_{\free}(x,X)$ at $C$ equals 
$-K_{X_{\sm}} \cdot C - 2$.

Consider again $f \in \Hom_{\bir}(\bP^1,X)$ with image $C$.
We say that $C$ is \emph{embedded} if $f$ is an immersion
into $X_{\sm}$. This is an open property in view of
\cite[I.1.10.1]{Kollar}, and hence the embedded free curves
form an open subscheme $\RatCurves_{\emfr}(X)$ of
$\RatCurves(X)$.

\begin{lemma}\label{lem:em}
For any $x \in X_{\sm}$, the natural map 
$\rho_x : \mu^{-1}(x) \to \RatCurves(X)$ restricts to an immersion 
on the smooth open subscheme consisting of embedded free curves. 
\end{lemma}

\begin{proof}
By the above discussion, we may view $\rho_x$ as
the natural map 
\[ \pi : \Hom_{\emfr}(\bP^1,X; 0 \to x)/\Aut(\bP^1,0)
\longrightarrow \Hom_{\emfr}(\bP^1,X)/\Aut(\bP^1) \]
with an obvious notation. 

We check that $\pi$ is injective on $k$-rational points. 
For any $f_1,f_2$ in  $\Hom_{\emfr}(\bP^1,X; 0 \to x)$
such that $f_2 = f_1 \circ \varphi$ for some
$\varphi \in \Aut(\bP^1)$, we have 
$x = f_2(0) = f_1(\varphi(0)) = f_1(0)$,
and hence $\varphi(0) = 0$ as desired.

Next, we check that the differential of $\pi$ 
at any $k$-rational point is injective. Let 
$f \in \Hom_{\emfr}(\bP^1,X; 0 \to x)$; then
we have a commutative diagram (with an obvious
notation again)
\[ \xymatrix{
0 \ar[r] & \Lie \Aut(\bP^1,0) \ar[r] \ar[d] ^{=}
& \Lie \Aut(\bP^1) \ar[r] \ar[d]^{=} & T_0 \bP^1 
\ar[r] \ar[d]^{=} & 0 \\
0 \ar[r] & H^0(\bP^1,T_{\bP^1}(-1)) \ar[r] \ar[d]
& H^0(\bP^1,T_{\bP^1}) \ar[r] \ar[d] & T_0 \bP^1  
\ar[r] \ar[d]^{df_0} & 0 \\
0 \ar[r] & H^0(\bP^1, f^*(T_{X_{\sm}})(-1)) \ar[r]  & 
H^0(\bP^1, f^*(T_{X_{\sm}}))  \ar[r] & T_x X \ar[r] & 0, \\
} \]
where $df_0$ is injective. So the induced map
\[ H^0(\bP^1, f^*(T_{X_{\sm}})(-1)))/ \Lie \Aut(\bP^1,0)
\longrightarrow 
H^0(\bP^1, f^*(T_{X_{\sm}}))/ \Lie \Aut(\bP^1) \]
is injective as well, as desired.
\end{proof}

\subsection{Families of rational curves}
\label{subsec:families}

The normal scheme $\RatCurves(X)$ is the disjoint union of open 
and closed normal quasi-projective varieties. These (connected 
or irreducible) components are called the 
\emph{families of rational curves on $X$}. 
Every such family $\cK$ comes with a universal family 
$\cU := \rho^{-1}(\cK) \to \cK$, where $\cU$ is a component 
of $\Univ(X)$. For any $x \in X$, we denote the fiber of 
$\mu : \cU \to X$ by $\cU_x$, and let $\cK_x := \rho(\cU_x)$; 
then the induced morphism $\rho_x : \cU_x \to \cK_x$ is finite.
The family $\cK$ is \emph{covering} if $\mu$ is dominant, i.e.,
$\cK_x$ (or equivalently $\cU_x$) is non-empty for a general
point $x$.  If in addition $\cK_x$ (or equivalently $\cU_x$)
is projective for $x$ general, then $\cK$ is called a
\emph{family of minimal rational curves}, or just a 
\emph{minimal family} for simplicity. Examples of minimal families
include the covering families of lines in some projective embedding 
of $X$; also, note that lines contained in the smooth locus yield
examples of embedded curves.

For any family of rational curves $\cK$ as above, there exists 
a unique irreducible component $Z$ of $\Hom_{\bir}(\bP^1,X)$ 
together with a principal $\Aut(\bP^1)$-bundle $Z^{\n} \to \cK$, 
where $Z^{\n}$ denotes the normalization. Let $Z_{\free}$ denote 
the smooth open subscheme of $Z$ consisting of free morphisms; 
then $Z_{\free}$ is stable under $\Aut(\bP^1)$, and hence we
obtain a principal $\Aut(\bP^1)$-bundle $Z_{\free} \to \cK_{\free}$,
where $\cK_{\free} \subset \cK$ denotes the smooth open 
subscheme of free curves. We may view the points of $\cK_{\free}$ 
as (possibly singular) rational curves in $X$. For any such
curve $C$, the fiber of $\rho$ at $C$ is a projective line, 
and the restriction of $\mu$ to this fiber yields the 
normalization map of $C$. Also, given $x \in X_{\sm}$,
the morphism $\rho_x : \cU_x \to \cK_x$ restricts to 
an isomorphism on the open subscheme $\cK_{\emfr,x}$ consisting
of embedded free curves (Lemma \ref{lem:em}). Moreover,
by assigning to each such curve its tangent direction
at $x$, we obtain a morphism
\begin{equation}\label{eqn:tangent}
\tau = \tau_{X,x} : \cK_{\emfr,x} \longrightarrow \bP(T_x X),
\end{equation}
where $\bP(T_x X)$ denotes the projectivization of
the tangent space. If $\charac(k) = 0$ and the point $x$ 
is general, then $\tau$ extends to a finite morphism from 
the normalization of $\cK_x$ (see \cite[Thm.~3.4]{Kebekus}).

\begin{remark}\label{rem:nofree}
Assume that $X$ is smooth. If $\charac(k) = 0$ then every
covering family contains a free curve, as follows from
\cite[II.3.10]{Kollar} together with generic smoothness. 
But this fails if $\charac(k) = p > 0$, as shown by the
following example adapted from \cite[V.1.4.3]{Kollar}. 
Consider the hypersurface $X$ in $\bP^2 \times \bP^2$ 
with homogeneous equation 
\[ x_0 y_0^p + x_1 y_1^p + x_2 y_2^p = 0,  \]
where $x_0, x_1,x_2$ (resp.~$y_0, y_1, y_2$) are homogeneous 
coordinates on the first (resp.~second) copy of $\bP^2$. 
Then $X$ is smooth, the geometric fibers of the first projection 
$\pr_1 : X \to \bP^2$ are all non-reduced, and their reduced 
subschemes are lines. For the corresponding family 
of rational curves, $\cU$ is the hypersurface in 
$\bP^2 \times \bP^2$ with equation 
$x_0 y_0 + x_1 y_1 + x_2 y_2 = 0$, and $\cK = \bP^2$;
in particular, $\cK$ is minimal. Also, the morphism $\rho$ is the 
first projection, and 
\[ \mu( [x_0: x_1: x_2], [y_0: y_1: y_2]) :=  
([x_0^p : x_1^p : x_2^p], [y_0: y_1: y_2]). \] 
In particular, all the geometric fibers of $\mu : \cU \to \cK$ are 
fat points of multiplicity $p$, and hence $\cK$ contains no free 
curve. Note that $X$ is homogeneous under an appropriate 
action of $\Aut(\bP^2)$, and the stabilizer of any $x \in X$ 
is not reduced (or equivalently, not smooth). So $X$ is 
a variety of unseparated flags in the sense of \cite{HL};
one can show that any such variety admits a minimal
family which contains no free curve.
\end{remark}

We now discuss covariance properties under a morphism 
$\pi : X \to Y$, where $Y$ is a projective variety. Let $\cK$ 
be a family of rational curves on $X$, and $Z$ 
the corresponding irreducible component of 
$\Hom_{\bir}(\bP^1,X)$. Assume that there exists $f \in Z$ 
such that the composition $\pi \circ f: \bP^1 \to Y$ is free. 
Then $\pi \circ f$ is a smooth point of a unique irreducible
component $W$ of $\Hom_{\bir}(\bP^1,Y)$, which defines
in turn a family of rational curves $\cL$ on $Y$. 
The composition of natural morphisms 
$Z^{\n} \to Z \to \Hom(\bP^1,X) \to \Hom(\bP^1,Y)$
is $\Aut(\bP^1)$-equivariant and its image contains
a smooth point of $W$. This yields a rational map
\[ \pi_* : \cK \dasharrow \cL, \]
which is defined on the open subset consisting of 
those free curves that are sent to free curves in $Y$.

In the opposite direction, given $\pi$ and $\cK$ as above, 
we say that $\pi$ \emph{contracts} some $C_0 \in \cK$ 
if the composition
$\rho^{-1}(C_0) \stackrel{\mu}{\longrightarrow}
X \stackrel{\pi}{\longrightarrow} Y$ 
is constant. Then $\pi$ contracts all $C \in \cK$: indeed, 
choose an ample line bundle $M$ on $Y$ and let
$L := \mu^* \pi^*(M)$. Then the degree of $L$ on 
$\rho^{-1}(C)$ is independent of $C$ (as follows e.g.~from 
\cite[Prop.~10.3]{Fulton}), and this degree is $0$
if and only if $C$ is contracted by $\pi$.

Based on this observation, we now describe the minimal
families in a product of varieties:

\begin{lemma}\label{lem:product}
Let $Y,Z$ be projective varieties, and $X := Y \times Z$
with projections $\pi : X \to Y$, $\varphi : X \to Z$.

\begin{enumerate}

\item[{\rm (i)}] The pull-back morphism 
\[ \pi^* : \Hom(\bP^1,Y) \times Z \longrightarrow \Hom(\bP^1,X), 
\quad (f,z) \longmapsto (t \mapsto (f(t),z)) \] 
induces a closed immersion 
$\RatCurves(Y) \times Z \to \RatCurves(X)$
with image a union of components. 

\item[{\rm (ii)}] $\pi^*$ sends covering 
(resp.~minimal) families to covering (resp.~minimal) families.

\item[{\rm (iii)}] A family of rational curves $\cK$ on $X$
is the pull-back of a family on $Y$ if and only if $\varphi$ 
contracts some curve in $\cK$.

\item[{\rm (iv)}] Every family of minimal rational curves on $X$ 
is the pull-back of a unique family of minimal rational curves 
on $Y$ or $Z$. 

\end{enumerate}

\end{lemma}

\begin{proof}
One may easily check that the ``constant'' morphism
\[ c : Z \longrightarrow \Hom(\bP^1,Z), \quad 
z \longmapsto (t \mapsto z) \] 
is a closed immersion; moreover, the diagram
\[ \xymatrix{
\Hom(\bP^1, Y) \times Z \ar[r] \ar[d]_{\pi^*} & Z \ar[d]^c \\
\Hom(\bP^1, X) \ar[r] & \Hom(\bP^1, Z)  \\
} \]
is cartesian, where the top horizontal arrow is the projection, and 
the bottom horizontal arrow is the composition with $\varphi$.
Thus, $\pi^*$ is a closed immersion as well. Also, $\pi^*$
sends $\Hom_{\bir}(\bP^1,Y) \times Z$ to $\Hom_{\bir}(\bP^1,X)$, 
and its image (considered with its reduced scheme structure)
consists of those morphisms $f$ such that 
$\deg \, \varphi^* f^*(M) = 0$ for a given ample line bundle $M$ 
on $Z$. This readily yields the assertions (i), (ii) and (iii).

(iv)  Let $\cK$ be a minimal family on $X$, and $x  \in X$ such 
that $\cU_x$ is non-empty and projective. Choose 
$f \in \Hom_{\bir}(\bP^1, X; 0 \mapsto x)$ with image $C$
such that $C$ is also the image of a curve in $\cU_x$. Write 
$x = (y,z)$ and $f = (g,h)$, where $g \in \Hom(\bP^1, Y; 0 \mapsto y)$ 
and $h \in \Hom(\bP^1, Z; 0 \mapsto z)$. We may view $f$ as the
composition 
\[ \bP^1 \stackrel{\delta}{\longrightarrow} \bP^1 \times \bP^1 
\stackrel{g \times h}{\longrightarrow} Y \times Z, \]
where $\delta$ denotes the diagonal embedding. 
If both $g$ and $h$ are non-constant, then $g \times h$ is
finite. As the image of $\delta$ degenerates in $\bP^1 \times \bP^1$
to a reducible curve through $(0,0)$, namely
$(\bP^1 \times \{ 0 \}) \cup (\{ 0 \} \times \bP^1)$, this yields 
a degeneration of $C$ in $X$ to a reducible curve through $x$. 
But this contradicts the minimality of $\cK$. Thus, we may assume 
that $g$ is constant; then (iii) implies that $\cK$ is the pull-back 
of a minimal family on $Y$. 
\end{proof}

\subsection{Almost homogeneous varieties}
\label{subsec:almost}

We now assume that the projective variety $X$ is equipped with 
an action of a connected algebraic group $G$. Then $G$ acts 
on $\Hom(\bP^1,X)$ via its action on $X$, which commutes
with the action of $\Aut(\bP^1)$ and stabilizes the open 
subscheme $\Hom_{\bir}(\bP^1,X)$. This yields actions of $G$
on $\Hom_{\bir}^{\n}(\bP^1,X)$, $\RatCurves(X)$, $\Univ(X)$
such that the diagram (\ref{eqn:hom}) is equivariant.
Also, $G$ acts on the Chow scheme and the diagram
(\ref{eqn:chow}) is equivariant as well. Thus, so is
the morphism (\ref{eqn:mu}). Since $G$ is connected,
every family of rational curves on $X$ is stable by $G$.

Next, assume that there exists a point $x \in X$ such 
that the orbit $X^0 = G \cdot x$ is open in $X$, i.e.,
$X$ is almost homogeneous under $G$. Then there exist
covering families of rational curves on $X$, since 
$G$ is a rational variety. We assume in addition that the
(scheme-theoretic) stabilizer $H = G_x$ is smooth and
connected.

Consider a family $\cK$ of rational curves on $X$. 
Let $\cU^0 := \mu^{-1}(X^0)$; 
this is an open $G$-stable subset of $\cU$. 
Since $\rho: \cU \to \cK$ is flat, $\rho(\cU^0) =: \cK^0$ 
is a $G$-stable open subset of $\RatCurves(X)$; it consists 
of those curves in $\cK$ that meet $X^0$. We also have
a smaller $G$-stable open subset $\cK(X^0)$,
consisting of those curves that are contained in $X^0$. 
This yields a commutative diagram of $G$-varieties
\begin{equation}\label{eqn:open} 
\xymatrix{
\cU(X^0) \ar[r] \ar[d] & \cU^0 \ar[r] \ar[d] &
\cU \ar[d]^{\rho} \\
\cK(X^0) \ar[r] & \cK^0 \ar[r] & \cK, \\
} \end{equation}
where the horizontal arrows are open immersions,
the left and right vertical arrows are $\bP^1$-bundles,
and the middle vertical arrow is smooth.

\begin{lemma}\label{lem:covering}
With the preceding notation and assumptions, $\cK$
is covering if and only if $\cU_x$ is non-empty; equivalently,
$\cK_x$ is non-empty. Under these assumptions, $\cU_x$ 
is a normal variety and $\cK_x$ is a variety.
\end{lemma}

\begin{proof}
The morphism $\mu$ restricts to a $G$-equivariant morphism 
$\mu^0: \cU^0 \longrightarrow X^0$ 
with fiber at $x$ being $\cU_x$. By identifying $X^0$ with 
the homogeneous space $G/H$, this yields a $G$-equivariant 
isomorphism $\cU^0 \simeq G \times^H \cU_x$,
and in turn a cartesian square
\[ \xymatrix{
G \times \cU_x \ar[r] \ar[d] & G \ar[d] \\
\cU^0 \ar[r] & G/H, \\
} \]
where the vertical arrows are principal $H$-bundles. 
Since $\cU^0$ is a normal variety and $H$ is smooth and
connected, it follows that $\cU_x$ is a normal variety as well.
Moreover, $\cU_x$ is non-empty if and only if so is $\cU^0$, 
or equivalently $\cK^0$; this completes the proof.
\end{proof}

\begin{lemma}\label{lem:smooth}
Let $\cK$ be a family of rational curves on $X$, containing 
a curve $C$ which consists of smooth points and meets $X^0$.

\begin{enumerate}

\item[{\rm (i)}]  $C$ is free.

\item[{\rm (ii)}] $\cK$ is covering and $\cU_x$ is a normal variety.

\item[{\rm (iii)}] $\cU_x$ is smooth at $C$, of dimension 
$- K_{X_{\sm}} \cdot C - 2$.

\item[{\rm (iv)}]  $\cU_x$ is isomorphic to a component of 
$\RatCurves(x,X)$.

\end{enumerate}

\end{lemma}

\begin{proof}
(i) The smooth locus $X_{\sm}$ is $G$-stable and contains the open orbit
$X^0 \simeq G/H$. Thus, the tangent bundle $T_{X_{\sm}}$ is equipped
with a space of global sections: the image of the Lie algebra $\fg$
of $G$. Since $H$ is smooth, this space generates $T_{X^0}$.
As a consequence, the induced map 
$\cO_{\bP^1} \otimes \fg \to f^*(T_{X_{\sm}})$
is generically surjective, where $f: \bP^1 \to C$
denotes the normalization. Using the fact that every vector
bundle on $\bP^1$ is a direct sum of line bundles, it follows
easily that $f^*(T_{X_{\sm}})$ is globally generated. 

(ii) This is a consequence of Lemma \ref{lem:covering}.

(iii) This follows from the properties of free morphisms 
recalled in \S \ref{subsec:spaces}. 

(iv) By (ii) and (\ref{eqn:bundle}), there is a principal 
$\Aut(\bP^1,0)$-bundle $\eta^{-1}(Y) \to \cU_x$ 
for some irreducible component $Y$ of 
$\Hom_{\bir}(\bP^1,X; 0 \mapsto x)$; moreover,
$\eta^{-1}(Y)$ is a normal variety. By the universal
property of the normalization, this yields a finite morphism
\[ \varphi : \eta^{-1}(Y) \to \Hom_{\bir}^{\n}(\bP^1,X; 0 \mapsto x), \]
which restricts to an isomorphism on the open subset of free morphisms. 
Thus, $\varphi$ yields an isomorphism to a component of 
$\Hom_{\bir}^{\n}(\bP^1,X; 0 \mapsto x)$. As $\varphi$ is 
$\Aut(\bP^1,0)$-equivariant, it descends to the desired isomorphism.
\end{proof}

Next, we consider covariance properties of covering families,
building on the observations after Remark \ref{rem:nofree}.
Let $X$ be as above, and $\pi : X \to Y$ a surjective morphism,
where $Y$ is a projective variety. Assume that $Y$ is equipped 
with a $G$-action such that $\pi$ is equivariant (this assumption holds 
if $\pi_* \cO_X = \cO_Y$ in view of Blanchard's lemma, see 
e.g.~\cite[\S 4.2]{BSU}). Let $y:= \pi(x)$ and $Y^0 := Gy$; then 
$Y^0= \pi(X^0)$ is open in $Y$. Finally, let $\cK$ be a covering 
family of rational curves on $X$.

\begin{lemma}\label{lem:rational}
Assume that there exists a curve $C \in \cK^0$ such that
$C \subset X_{\sm}$, $\pi \vert_C$ is birational onto its
image $D$, and $D \subset Y_{\sm}$. Then $D \in \cL$
for a unique covering family $\cL$ of rational curves on $Y$.
Moreover, $\pi$ induces a $G$-equivariant rational map 
$\pi_* : \cK \dasharrow \cL$ 
which is defined at $C$ and such that $\pi_*(C) = D$.
\end{lemma}

\begin{proof}
The assumptions make sense, since every rational curve on $X$
meeting $X^0$ and contained in $X_{\sm}$ is free 
(Lemma \ref{lem:smooth}). The statement follows readily
from the discussion at the end of \S \ref{subsec:families}, with the
exception of the equivariance of $\pi_*$ which is easily checked.
\end{proof}

\begin{remark}\label{rem:rational}
Under the assumptions of the above lemma, $\pi_*$
restricts to an $H$-equivariant rational map 
$\pi_* = \pi_{*,x} : \cK_x \dasharrow \cL_y$.
(Indeed, replacing $C$ with the translate $gC$ for some 
$g \in G$, we may assume that $x \in C$, and hence 
$y \in D$). Moreover, we have a commutative diagram 
of rational maps
\begin{equation}\label{eqn:tangentbis} 
\xymatrix{
\cK_x \ar@{-->}[r] \ar@{-->}[d]_{\pi_*} & \bP(T_x X) 
\ar@{-->}[d]^{d\pi_x} \\
\cL_y \ar@{-->}[r] & \bP(T_y Y), \\
} \end{equation}
where the horizontal arrows arise from the tangent maps
(\ref{eqn:tangent}).

The assumptions of Lemma \ref{lem:rational} hold 
if $\pi$ is birational and $\cK(X^0)$ is non-empty; 
then $\pi_*$ restricts to an isomorphism 
$\cK(X^0) \to \cL(X^0)$, and hence is birational.
The assumptions also hold when $\pi$ is birational and 
$X,Y$ are smooth; then $\pi_*$ is an injective morphism. 
\end{remark}

\begin{lemma}\label{lem:contraction}
Let $I := G_y \supset G_x = H$ and $F := \pi^{-1}(y)$; 
assume that $I$ is smooth and connected. 
Then $F$ is a projective variety equipped with an action 
of $I$ and having an open orbit $F^0 = I x \simeq I/H$.

If in addition $\pi$ contracts some curve in $\cK$, then 
$\cK_x = \cL_x$ for a unique covering family of rational
curves $\cL$ on $F$. Moreover, $\cK^0 = G \cL^0$,
and $\cK$ is minimal if and only if so is $\cL$.
\end{lemma}

\begin{proof}
Note that $\pi^{-1}(Y^0)$ is a $G$-stable open subvariety
of $X$; in particular, it contains the open orbit $X^0$.
Moreover, $\pi$ restricts to a $G$-equivariant morphism 
$\pi^{-1}(Y^0) \to Y^0 \simeq G/I$ 
with fiber $F$ at $y$. This yields a $G$-equivariant
isomorphism $\pi^{-1}(Y^0) \simeq G \times ^I F$,
and in turn the first assertion by arguing as in the proof
of Lemma \ref{lem:covering}.

If $\pi$ contracts some curve in $\cK$, then it
contracts all curves in $\cK$, as seen at the end
of \S \ref{subsec:families}. Thus, every curve in 
$\cK_x$ is a rational curve on $F$. The second
assertion follows readily from this.
\end{proof}

\section{Minimal rational curves on flag varieties}
\label{sec:flag}

\subsection{Flag varieties}
\label{subsec:flag}

Let $X$ be a projective variety, homogeneous under
the action of a connected linear algebraic group $G$.
Choose $x \in X$ and assume that the stabilizer
$G_x$ is smooth. Then $G_x$ is a parabolic subgroup 
of $G$, and hence is connected. Moreover, replacing
$G$ with its largest semi-simple quotient and then with
its simply-connected cover, we may and will assume that
$G$ is semi-simple and simply-connected. We identify 
$X$ with the homogeneous space $G/P$, where $P := G_x$.
The Lie algebras of $G, G_x, P, \ldots$ will be denoted by
$\fg,\fg_x,\fp, \ldots$

Choose a Borel subgroup $B \subset P$ and
a maximal torus $T \subset B$. Let $R$ denote
the root system of $(G,T)$, and $R^+$ the subset
of roots of $(B,T)$; then $R^+$ is a set of positive
roots of $R$. Denote by $R^-$ the corresponding
set of negative roots, and by 
$S = \{ \alpha_1, \ldots, \alpha_n \} \subset R^+$
the set of simple roots. The Weyl group
$W = N_G(T)/T$ is generated by the associated
simple reflections $s_1, \ldots,s_n$. For any 
$w \in W$, we denote by $\dot w \in N_G(T)$ a 
representative. Also, for any $\beta \in R$, we denote
by $U_{\beta} \subset G$ the corresponding root 
subgroup. Let $G_{\beta} \subset G$ denote the
subgroup generated by $U_{\beta}$ and $U_{- \beta}$;
then $G_{\beta}$ is a closed semi-simple subgroup 
of rank $1$, normalized by $T$. For any $w \in W$, 
the conjugation by $\dot w$ sends $U_{\beta}$ to 
$U_{w(\beta)}$, and $G_{\beta}$ to $G_{w(\beta)}$.

We also have the coroot system $R^{\vee}$ with simple 
roots $\alpha_1^{\vee}, \ldots,\alpha_n^{\vee}$; these
form a basis of the cocharacter lattice $X_*(T)$.
The dual basis of the character lattice $X^*(T)$
consists of the fundamental weights
$\varpi_1,\ldots, \varpi_n$. More intrinsically,
for any simple root $\alpha$, we will denote by
$\varpi_{\alpha}$ the fundamental weight with
value $1$ at $\alpha^{\vee}$, and $0$ at all
other simple coroots. Let 
$\rho := \varpi_1 + \cdots + \varpi_n$; then 
$\rho = \frac{1}{2} \sum_{\alpha \in R^+} \alpha$.
The height of any $\beta \in R^{\vee}$ is 
$\he(\beta) := \langle \rho, \beta \rangle$;
this equals the sum of the coordinates of $\beta$
in the basis of simple coroots.

Consider the Levi decomposition $P = R_u(P) L$,
where $L$ is a connected reductive subgroup of 
$G$ containing $T$; then $B_L := B \cap L$
is a Borel subgroup of $L$. Denote by 
$R_L \subset R$ the root system of $(L,T)$,
with subset of positive roots $R^+_L = R_L \cap R^+$ 
and subset of simple roots $I := R_L \cap S$. Then $P$ is
generated by $B$ and the $\dot s_{\alpha}$, where
$\alpha \in I$; we write $P = P_I$ and $L = L_I$.

The character group $X^*(P)$ is identified via restriction
to the subgroup of $X^*(T)$ with basis the 
$\varpi_{\alpha}$, where $\alpha \in S \setminus I$. 
Also, every $\lambda \in X^*(P)$ defines the associated line
bundle $\cL_{G/P}(\lambda)$ on $G/P$; moreover,
$\cL_{G/P}(\lambda)$ is ample if and only if $\lambda$ 
has positive coordinates in the above basis. 
The assignement
$\lambda \mapsto \cL_{G/P}(\lambda)$ yields an
isomorphism $X^*(P) \simeq \Pic(G/P)$. In particular,
$G/P$ has a smallest ample line bundle, namely
$\cL_{G/P}(\varpi)$, where 
$\varpi = \varpi_I := 
\sum_{\alpha \in S \setminus I} \varpi_{\alpha}$.
Every ample line bundle on $G/P$ is very ample,
and hence defines a projective embedding
$G/P \hookrightarrow \bP(V(\lambda))$,
where $V(\lambda) := H^0(G/P,\cL_{G/P}(\lambda))^*$,
and $\bP(V)$ denotes the projective space
of lines in a vector space $V$. This embedding is
equivariant for the natural action of $G$ on $G/P$,
and its linear representation in $V(\lambda)$
(a highest weight module, see e.g. \cite[II.2.13]{Jantzen}). 
In particular, we have a ``smallest'' projective embedding
\begin{equation}\label{eqn:projective} 
G/P \hookrightarrow \bP(V(\varpi)). 
\end{equation}
Also, recall that the canonical class of $X$ satisfies 
\begin{equation}\label{eqn:canonical} 
\cO(K_X) \simeq \cL_{G/P}(- 2 (\rho - \rho_I)).
\end{equation}

The parabolic subgroup $P$ is maximal if and only
if $I$ is the complement of a unique simple root
$\alpha$. We then write $P = P^{\alpha}$.
More generally, we will use the notation 
$P^{S \setminus I}$ for $P_I$ whenever this is convenient.

\subsection{Schubert varieties}
\label{subsec:schubert}

We keep the notation and assumptions of the previous
subsection. The Weyl group $W_L = N_L(T)/T$ 
is generated by the simple reflections $s_{\alpha}$, where 
$\alpha \in I$; we also denote this group by $W_I$. 
Let $W^I$ denote the subset of $W$ consisting of those
$w$ such that $w(\alpha) \in R^+$ for all $\alpha \in I$;
equivalently, $R^+_I \subset w^{-1}(R^+)$.
Then $W^I$ is a set of representatives of the coset space
$W/W_I$, consisting of the elements of minimal length in 
their right coset (for the length function $\ell$ on $W$
relative to the generators $s_1, \ldots, s_n$).
Note that $w \in W^I$ has length $1$ if and only if 
$w = s_{\alpha}$ for some $\alpha \in S \setminus I$.
On the other hand, the unique element of maximal
length in $W^I$ is $w_0 w_{0,I}$, where $w_0$ 
(resp.~$w_{0,I}$) denotes the longest element of $W$ 
(resp.~$W_I$).

For any $w \in W$, the point ${\dot w} x \in G/P$ is
independent of the choice of the representative $\dot w$;
we thus denote this point by $wx$. Recall that the $wx$,
where $w \in W^I$, are exactly the $T$-fixed points in $G/P$;
moreover, $G/P$ is the disjoint union of the $B$-orbits 
$B wx$. The stabilizer $B_{wx}$ is generated
by $T$ and the root subgroups $U_{\beta}$, where 
$\beta \in R^+ \cap  w(R^+)$; in particular, 
$B_{wx}$ is smooth and connected. 
The closure of $Bwx$ in $G/P$ is the Schubert variety 
$X(w)$; we have $\dim(X(w)) = \dim(Bwx) = \ell(w)$. 
In particular, the Schubert varieties of dimension $1$ are exactly 
the $X(s_{\alpha})$, where $\alpha \in S \setminus I$. Note that 
$X(s_{\alpha}) = P_{\alpha} x = L_{\alpha} x = G_{\alpha} x
= \overline{U_{-\alpha} x}$ 
is a $T$-stable curve in $G/P$ with fixed points $x, s_{\alpha}x$. 

We now collect some basic properties of $T$-stable curves,
which are essentially known (see \cite[\S 3, \S 4]{FW}); we will
provide proofs for completeness.

\begin{lemma}\label{lem:curves}

\begin{enumerate}

\item[{\rm (i)}] The $T$-stable curves in $G/P$ through 
the $T$-fixed point $wx$ are exactly the curves
\[ C_{w,\beta} := G_{\beta} wx = \overline{U_{-\beta}wx}, \]
where $\beta \in w(R^+ \setminus R^+_I)$
(so that $X(s_{\alpha}) = C_{1,\alpha}$ for any 
$\alpha \in S \setminus I$).

\item[{\rm (ii)}] The $T$-fixed points in $C_{w,\beta}$
are exactly $wx$ and $s_{\beta}wx$. 

\item[{\rm (iii)}] $C_{w,\beta} \simeq \bP^1$.

\item[{\rm (iv)}] For any $\lambda \in X^*(P)$, we have
\[ \cL_{G/P}(\lambda) \cdot C_{w,\beta} = 
\langle \lambda, w^{-1}(\beta^{\vee}) \rangle. \]

\item[{\rm (v)}] We have
\[ -K_{G/P} \cdot C_{w,\beta} = 
2 \langle \rho - \rho_I, w^{-1}(\beta^{\vee}) \rangle  =
\he(w^{-1}(\beta^{\vee})) + \he(w_{0,I}w^{-1}(\beta^{\vee})). \]
In particular, 
$-K_{G/P} \cdot X(s_{\alpha}) = 
\he(w_{0,I}(\alpha^{\vee})) + 1$.

\end{enumerate}

\end{lemma}

\begin{proof}
Using the action of $N_G(T)$ on $G/P$ which yields 
a transitive action of $W$ on $T$-fixed points, we
may reduce to the case where $w = 1$. 

By the Bruhat decomposition, we have a $T$-equivariant
open immersion
\[ \prod_{\beta \in R^+ \setminus R^+_I} U_{- \beta}
\longrightarrow G/P, \quad
(g_{\beta}) \longmapsto (\prod_{\beta} g_{\beta}) x, \]
where the product is taken in any order. Thus, $x$
has an open $T$-stable neighborhood in $G/P$, 
isomorphic to an affine space on which $T$ acts
linearly with weights being the $-\beta$, where 
$\beta \in R^+ \setminus R^+_I$; moreover, 
each such weight has multiplicity $1$. It follows that
the $T$-stable curves in $G/P$ through $x$ are
exactly the closures of the $T$-stable lines in this
neighborhood, i.e., of the orbits $U_{-\beta} x$;
moreover, the orbit maps 
$U_{- \beta} \to U_{- \beta} x$ are isomorphisms.
Since $x$ is fixed by $G_{\beta} \cap B$, a Borel 
subgroup of $G_{\beta}$, we have
\[ C_{1,\beta} = \overline{U_{-\beta} x} = G_{\beta} x 
\simeq G_{\beta}/(B \cap G_{\beta}). \]
This implies (i), (ii) and (iii).

(iv) The restriction to $C_{1,\beta}$ of the $G$-linearized 
line bundle $\cL_{G/P}(\lambda)$ is the $G_{\beta}$-linearized 
line bundle $\cL_{G_{\beta}/(B \cap G_{\beta})}(\lambda)=: \cL$. 
Moreover, $T \cap G_{\beta}$
is a maximal torus of $G_{\beta}$, the image of the coroot
$\beta^{\vee} : \bG_m \to T$; the scheme-theoretic
center of $G_{\beta} \simeq \SL(2)$ is the image of 
the $2$-torsion subgroup scheme $\mu_2 \subset \bG_m$
under $\beta^{\vee}$. 
Also, $\beta^{\vee}$ acts on the fiber of $\cL$ at $x$ 
(resp.~$s_{\beta} x$) via the weight
$\langle \lambda, \beta^{\vee} \rangle$ 
(resp.~$\langle s_{\beta}(\lambda), \beta^{\vee} \rangle
= - \langle \lambda, \beta^{\vee} \rangle$). 
Identifying $C_{1,\beta}$ with $\bP^1$, it follows easily
that $\cL \simeq \cO_{\bP^1}(n)$, where 
$n := \langle \lambda, \beta^{\vee} \rangle$.

(v) The first equality follows directly from (iv) in view of
the isomorphism (\ref{eqn:canonical}). 
As $\rho - w_{0,I}(\rho) = 2 \rho_I$, we obtain
$2(\rho - \rho_I) = \rho + w_{0,I}(\rho)$; this implies
the second equality. 
\end{proof}

By Lemma \ref{lem:curves}, every $1$-dimensional 
Schubert variety $X(s_{\alpha})$ satisfies 
\[ \cL_{G/P}(\varpi) \cdot X(s_{\alpha}) =
\langle \varpi, \alpha^{\vee} \rangle = 1, \]
that is, $X(s_{\alpha})$ is a line in $\bP(V(\varpi))$.
We thus say that $X(s_{\alpha})$ is a 
\emph{Schubert line}. 

Also, note that every $T$-stable curve in $G/P$
is a line if and only if we have
$\langle \varpi, \beta^{\vee} \rangle \leq 1$
for all $\beta \in R^+$, i.e., the dominant weight
$\varpi$ is \emph{minuscule}. Then the parabolic 
subgroup $P$ is also called minuscule; it is maximal 
if $G$ is simple. The projective homogeneous space
$G/P$ is called minuscule as well. Moreover, the weights 
of $T$ in $V(\varpi)$ are exactly the $w(\varpi)$, where 
$w \in W$; as a consequence, the $G$-module $V(\varpi)$
is simple.

\subsection{Lines on flag varieties}
\label{subsec:lines}

We still keep the notation and assumptions of 
\S \ref{subsec:flag}, and consider a minimal family
 $\cK$ on $X = G/P$.

\begin{lemma}\label{lem:lines}

\begin{enumerate}

\item[{\rm (i)}] $\cK$ consists of free curves.

\item[{\rm (ii)}] $\cK_x$ is a smooth projective variety
containing a unique Schubert line $X(s_{\alpha})$,
where $\alpha \in S \setminus I$. Moreover,
$\dim(\cK_x) = \he(w_{0,I}(\alpha^{\vee})) - 1$.

\item[{\rm (iii)}] $\cK_x$ consists of the lines in the orbit
$P_{I \cup \{ \alpha \} } x$ through $x$.

\item[{\rm (iv)}] $\cK$ consists of those lines in $G/P$ 
that are contracted by the natural morphism 
$\pi_{\alpha} : G/P = G/P_I \to  G/P_{I \cup \{ \alpha \} }$.

\end{enumerate}

\end{lemma} 

\begin{proof}
(i) This follows from Lemma \ref{lem:smooth}.

(ii) The scheme $\cK_x$ is projective by assumption.
It is smooth in view of (i) and \S \ref{subsec:families}, 
and irreducible by Lemma \ref{lem:covering}. 
Moreover, $\cK_x$ is equipped with an action of
$P = G_x$. By Borel's fixed point theorem, it follows
that $\cK_x$ contains a $B$-fixed point, and hence
a Schubert line $X(s_{\alpha})$. The assertion on the
dimension follows by combining Lemmas \ref{lem:smooth}
and \ref{lem:curves} (v).

The morphism $\pi_{\alpha}$ contracts the Schubert
line $X(s_{\alpha})$ and sends any other Schubert line
isomorphically to its image. As a consequence, every 
curve in $\cK$ is contracted by $\pi_{\alpha}$, and 
$\cK_x$ contains no other Schubert line.

(iii) and (iv) By (i), the lines in $G/P$ form a disjoint union 
of minimal families. Thus, $\cK$ is the unique family of lines 
such that $\cK_x$ contains $X(s_{\alpha})$. This yields the 
assertions by using Lemma \ref{lem:contraction}.
\end{proof}

With the notation of Lemma \ref{lem:lines}, we have 
\[ P_{I \cup \{ \alpha \} } x \simeq 
P_{I \cup \{ \alpha \} }/P_I \simeq 
L_{I \cup \{ \alpha \} }/(P_I \cap L_{I \cup \{ \alpha \} } ). \]
Moreover, the scheme-theoretic intersection 
$P_I \cap L_{I \cup \{ \alpha \} }$ is smooth (as follows from
\cite[13.21]{Borel}), and hence is a maximal parabolic 
subgroup of the connected reductive group $L_{I \cup \{ \alpha \} }$.
So this lemma reduces the description of minimal families 
on $G/P$ to the case where $P$ is maximal; then there is 
a unique such family, and it consists of the lines in 
$G/P \subset \bP(V(\varpi_{\alpha}))$, where $P = P^{\alpha}$. 
We may further assume that $G$ is simple; if in addition 
the simple root $\alpha$ is long, then we have the following result, 
which is known over the complex numbers (see 
\cite[Prop.~1]{HwM} and \cite[Thm.~4.8]{LM}):

\begin{proposition}\label{prop:long}
Let $P = P^{\alpha}$, where $\alpha$ is a long simple root.
Denote by $\cL$ the family of lines in $G/P$, and by $\cL_x$
the subfamily of lines through $x$. 

\begin{enumerate}
 
\item[{\rm (i)}] $\cL_x$ is the $L$-orbit of the Schubert
line $X(s_{\alpha})$. 

\item[{\rm (ii)}] The tangent map (\ref{eqn:tangent})
yields an immersion of $\cL_x$ into $\bP(T_x X)$.

\item[{\rm (iii)}] The (scheme-theoretic) stabilizer of 
$X(s_{\alpha})$ in $L$ is the parabolic subgroup 
$L \cap P_{\alpha^{\perp}}$, where 
$\alpha^{\perp} := \{ \beta \in S ~\vert~ 
\langle \beta, \alpha^{\vee} \rangle = 0 \}$.
 
\end{enumerate}

\end{proposition}

\begin{proof}
By \S \ref{subsec:spaces} and Lemma \ref{lem:lines} (i),
we have $\dim(\cL_x) = - K_X \cdot X(s_{\alpha}) - 2$.
Using Lemma \ref{lem:curves} (v), this yields
\[ 
\dim(\cL_x) = 2  \langle \rho - \rho_I, \alpha^{\vee} \rangle -2
= -2 \langle \rho_I, \alpha^{\vee} \rangle 
= - \sum_{\beta \in R^+_I} \langle \beta,\alpha^{\vee} \rangle.
\]
Note that $0 \leq - \langle \beta, \alpha^{\vee} \rangle \leq 1$
for all $\beta \in R^+_I$, since $\alpha$ is a long simple root
and differs from all the simple roots occuring in $\beta$.
As a consequence, 
\begin{equation}\label{eqn:dim} 
\dim(\cL_x) = \#(R^+_I \setminus R^+_{I \cap \alpha^{\perp}}).
\end{equation}

Next, observe that the tangent map yields a morphism
$\tau : \cL_x \to \bP(T_x X)$, since $\cL_x$ consists of
embedded free curves. Also, $T_x X \simeq \fg/\fp$
and this identifies $\tau(X(s_{\alpha}))$ with 
$[\fg_{-\alpha}]$, the image of the root subspace
$\fg_{-\alpha} \subset \fg$ in $\bP(\fg/\fp)$. Since
$\tau$ is $P$-equivariant, we have the inclusion of
stabilizers 
$L_{X(s_{\alpha})} \subset L_{[\fg_{-\alpha}]}$.
We now show that
\begin{equation}\label{eqn:stab}
L_{X(s_{\alpha})} = L_{[\fg_{-\alpha}]} 
= L \cap P_{\alpha^{\perp}}.
\end{equation}
The Lie algebra $\fl_{[\fg_{-\alpha}]}$ of 
$L_{[\fg_{-\alpha}]}$ is a subalgebra of $\fl$
containing the Borel subalgebra $\fb \cap \fl$, 
and hence is generated by $\fb \cap \fl$ and
the $\fg_{-\beta}$, where $\beta \in I$ and
$\fg_{-\beta}$ stabilizes $[\fg_{-\alpha}]$.
The latter condition is equivalent to 
$[ \fg_{- \beta}, \fg_{-\alpha} ] 
\subset \fg_{-\alpha} + \fp$. 
But 
$[ \fg_{- \beta}, \fg_{-\alpha} ] = 
\fg_{- \alpha - \beta}$ 
by a result of Chevalley (see e.g.~\cite[25.2]{Humphreys}), 
and $\fg_{- \alpha - \beta} = 0$ if $- \alpha - \beta$ 
is not a root. In any case, $- \alpha - \beta$ 
is not a root of $P$. Thus, the Lie algebra
$\fl_{[\fg_{-\alpha}]}$ is generated by $\fb \cap \fl$ 
and the $\fg_{- \beta}$, where $\beta \in I$ and  
$[ \fg_{- \beta}, \fg_{-\alpha} ] = 0$; equivalently,
$\beta \in \alpha^{\perp}$. In other terms,
$\fl_{[\fg_{-\alpha}]} = \fl \cap \fp_{\alpha^{\perp}}$.
On the other hand, $L \cap P_{\alpha^{\perp}}$  
stabilizes $X(s_{\alpha})$ and hence $[\fg_{- \alpha}]$.
Thus, 
\[ L \cap P_{\alpha^{\perp}} \subset L_{X(s_{\alpha})} 
\subset L_{[\fg_{-\alpha}]} \]
and equality holds for the corresponding Lie algebras. 
Since $L \cap P_{\alpha^{\perp}}$ is smooth, this 
easily implies the equalities (\ref{eqn:stab}). 
In turn, this yields (iii) and also the inequalities 
\[ \dim(\cL_x) \geq \dim(L X(s_{\alpha})) =
\dim(L_I/(L_I \cap P_{\alpha^{\perp}}))
= \#(R^+_I \setminus R^+_{I \cap \alpha^{\perp}}). \]
By (\ref{eqn:dim}), it follows that 
$\cL_x = L X(s_{\alpha})$, proving (i). Finally,
(ii) follows from (i) and (\ref{eqn:stab}).
\end{proof}

\begin{remarks}\label{rem:simplylaced}
We still assume that $G$ is simple and $P = P^{\alpha}$.

(i) Assume in addition that $G$ is \emph{simply-laced}, i.e., 
all the roots have the same length. Then $\alpha$ is long 
and $- \alpha$ is a minuscule weight of $L$; we denote by
$V_L(- \alpha)$ the corresponding highest weight module.
By Proposition \ref{prop:long} and its proof, $\cL_x$ is 
a minuscule homogeneous space under $L$; moreover, 
the tangent map is an immersion with image being 
the closed $L$-orbit in 
$\bP(V_L(-{\alpha})) \subset \bP(T_x X)$. One can show 
that the $T$-weights of $V_L(- \alpha)$ are exactly the
$-\beta$, where $\beta \in R^+$ and 
$\beta = \alpha + 
\sum_{\alpha_i \in S, \alpha_i \neq \alpha} n_i \alpha_i$
for some non-negative integers $n_i$; these form
a unique orbit of $W_L$. Moreover, $P$ is minuscule 
if and only if $V_L(- \alpha) = T_x X$; equivalently, 
$W_L$ acts transitively on $R^- \setminus R_L$. 

On the other hand, if $G$ is not simply-laced, then
there exists a short root $\alpha$ such that the variety of
lines through $x$ in $G/P^{\alpha}$ is not homogeneous 
under $P^{\alpha}$; see the main theorem of \cite{CC} 
for a more specific result, valid over an arbitrary field, and 
\cite{Strickland, LM} for further developments.

(ii) When $G$ is simply-laced, the minuscule homogeneous 
spaces $G/P$ and their varieties of lines are exactly those 
in the following table:

\[ \begin{tabular}{| c | c | c | c | c |}
\hline
Type of $G$ & $\alpha$ & $X$ & Type of $L$ &  $\cL_x$ \\
\hline
$\rA_n$ & $\alpha_m$ & $\bG(m,n+1)$ & $\rA_{m-1} \times \rA_{n-m} $
& $(\bP^{m-1})^* \times \bP^{n-m}$ \\
\hline
$\rD_n$ & $\alpha_1$ & $\bQ^{2n-2}$ & $\rD_{n-1}$ & $\bQ^{2n-4}$ \\
\hline
$\rD_n$ & $\alpha_{n-1}, \alpha_n$ & $\bS^{n(n-1)/2}$ & $\rA_{n-1}$ &  $\bG(2,n)$ \\
\hline 
$\rE_6$ & $\alpha_1, \alpha_6$ & $\bX^{16}$ & $\rD_5$ & $\bS^{10}$ \\
\hline
$\rE_7$ & $\alpha_7$ & $\bX^{27}$ & $\rE_6$ &  $\bX^{16}$ \\
\hline
\end{tabular} \]
Here $\bG(m,n+1)$ denotes the Grassmannian of $m$-dimensional 
linear subspaces of $k^{n+1}$, and $\bQ^n \subset \bP^{n+1}$ 
the $n$-dimensional smooth quadric. Also, $\bS^{n(n-1)/2}$ stands for 
the spinor variety of the corresponding dimension, and 
$\bX^{16}, \bX^{27}$ are two exceptional varieties of 
the corresponding dimensions again. The simple roots are ordered 
as in \cite[Chap.~VI]{Bourbaki}, and the list of minuscule
weights is taken from [loc.~cit., \S 4, Exerc.~15].

When $G$ is not simply-laced, one obtains in addition the pairs 
$(\rB_n, \alpha_n)$ and $(\rC_n,\alpha_1)$. The associated minuscule
varieties are isomorphic to those of the pairs $(\rD_{n+1},\alpha_n)$
and $(\rA_{2n-1},\alpha_1)$ respectively, that is, $\bS^{n(n+1)/2}$,
resp.~$\bP^{2n-1}$. Moreover, this identifies the Schubert varieties 
of the former pairs to those of the latter ones. Thus, we may assume 
that $G$ is simply-laced when studying Schubert varieties 
in minuscule $G$-homogeneous spaces.
\end{remarks}

\subsection{Lines on Schubert varieties}
\label{subsec:linesbis}

We keep the notations and assumptions of \S \S \ref{subsec:flag}
and \ref{subsec:schubert}, and start with the following observation:

\begin{lemma}\label{lem:curvesbis}
The $T$-stable curves in $X(w)$ through $wx$ are exactly the
$C_{w,\beta}$, where $\beta \in w(R^+) \cap R^-$.
\end{lemma}

\begin{proof}
The Bruhat decomposition yields a $T$-equivariant open immersion
\[ \prod_{\beta \in R^+ \cap w(R^-)} U_{\beta}
\longrightarrow X(w), \quad
(g_{\beta}) \longmapsto (\prod_{\beta} g_{\beta}) w x \]
with image $Bwx$, where the product is taken in any order.
The assertion follows from this by arguing as in the proof
of Lemma \ref{lem:curves}.
\end{proof}

Our next result implies that every Schubert variety is 
covered by translates of Schubert lines (see also 
\cite[Prop.~3.1]{HoM}):

\begin{lemma}\label{lem:covered}
The following are equivalent for $w \in W^I$ and 
$\alpha \in S \setminus I$:

\begin{enumerate}

\item[{\rm (i)}] $w(\alpha) \in R^-$.

\item[{\rm (ii)}] $X(w)$ is covered by $G$-translates of the Schubert
line $X(s_{\alpha})$.

\end{enumerate}

\end{lemma}

\begin{proof}
(i) $\Rightarrow$ (ii) The translate $w X(s_{\alpha})$ is a $T$-stable
curve with $T$-fixed points $wx, ws_{\alpha} x$.  By Lemma
\ref{lem:curves} or a direct argument, it follows that 
$w X(s_{\alpha}) = C_{w, w(\alpha)}$. So $w X(s_{\alpha}) \subset X(w)$ 
in view of Lemma \ref{lem:curvesbis}. We conclude that the translates
$bw X(s_{\alpha})$, where $b \in B$, cover $X(w)$. 

(ii) $\Rightarrow$ (i) By assumption, there exists $g \in G$
such that $g X(s_{\alpha})$ meets $Bwx$ and is contained in
$X(w)$; equivalently, $Bwx \cap g X(s_{\alpha})$ is dense in 
$g X(s_{\alpha})$. The Bruhat decomposition yields that
$g = b {\dot v} b'$ for some $b,b' \in B$ and $v \in W$; then
$Bwx \cap v X(s_{\alpha})$ is also dense in $v X(s_{\alpha})$. 
Since $Bwx \cap v X(s_{\alpha})$ is closed in $Bwx$ and
stable by $T$, it contains $wx$. Thus, $v X(s_{\alpha})$ is 
a $T$-stable curve through $wx$ in $X(w)$; in particular,
we have either $wx = vx$ or $wx = v s_{\alpha} x$. Replacing
$v$ with $v s_{\alpha}$, we may assume that $v(\alpha) \in R^-$;
then $v s_{\alpha} < v$ for the Bruhat order in $W$, and hence
in $W/W_I$. So we must have $wx = vx$, i.e., $w = v u$
for some $u \in W_I$. Then $w u^{-1}(\alpha) \in R^-$;
as $w(R^+_I) \subset R^+$, it follows that $w(\alpha) \in R^-$ 
as well.
\end{proof}

Next, assume that $P = P^{\alpha}$ where $\alpha$ is a long 
simple root. Let $\cK$ be a family of lines on $X(w)$,
and $\cK_{wx}$ the subfamily of lines through $wx$. It will be
convenient to consider the translate ${\dot w}^{-1} \cK$, 
a family of lines in $w^{-1} X(w)$ through the base point $x$. 
Recall from Proposition \ref{prop:long} that the family $\cL_x$ 
of lines in $G/P$ through $x$ is the minuscule variety 
$L X(s_{\alpha}) \simeq L/(L \cap P_{\alpha^{\perp}})$.

\begin{proposition}\label{prop:linesbis}
With the preceding notation and assumptions, $w^{-1}\cK_{wx}$
is a Schubert subvariety of $\cL_x$. Moreover, the
Schubert subvarieties obtained in this way are exactly the
$\overline{B_L v X(s_{\alpha})}$ where $v \in W_L$,
$wv(\alpha) \in R^-$, and $v$ is maximal for this property.
\end{proposition}

\begin{proof}
By Lemma \ref{lem:covering}, $\cK_{wx}$ is a projective variety 
equipped with an action of $B \cap w P w^{-1}$. Thus, 
$w^{-1} \cK_{wx}$ is a projective variety equipped with an action 
of $w^{-1} B w \cap P$, and hence of $B_L$. Moreover, there is
a quasi-finite equivariant morphism 
\[ \gamma : w^{-1} \cK_{wx} \longrightarrow \cL_x, \] 
since we may view $\cL_x$ as the Chow variety of lines in $G/P$.
Since $\cL_x$ is a flag variety under $L$, it has only finitely
many $B_L$-orbits by the Bruhat decomposition; also, the
$B_L$-isotropy groups are smooth and connected. As a consequence, 
$w^{-1} \cK_{wx}$ contains an open orbit of $B_L$; moreover, 
the stabilizer of a point $C$ of this orbit is contained 
in the stabilizer $(B_L)_{\gamma(C)}$, and both have the same 
dimension. Since $(B_L)_{\gamma(C)}$ is smooth and connected, 
both stabilizers must be equal and hence $\gamma$ is birational. 
Using the normality of Schubert varieties, it follows that 
$\gamma$ is an isomorphism. This proves the first assertion. 

For the second assertion, recall that every $B_L$-orbit in $\cL_x$
contains a unique $T$-fixed point; moreover, these fixed points
are exactly the $v X(s_{\alpha})$, where $v \in W_L$. 
Also, $v X(s_{\alpha}) = C_{1,v(\alpha)}$ as $v(\alpha) \in R^+$. 
By Lemma \ref{lem:curvesbis}, it follows that 
$v X(s_{\alpha}) \subset w^{-1} X(w)$ if and only if
$wv(\alpha) \in R^-$.
\end{proof}

\begin{remark}\label{rem:covered}
Let $P = P^{\alpha}$ as above and assume in addition that $P$
is minuscule. Then every minimal family $\cK$ in $X(w)$
consists of lines. Indeed, $\cK_{wx}$ is a projective variety
equipped with a $T$-action, and hence contains a $T$-fixed
point; also, every $T$-stable curve is a line.

Further, note that the smooth locus of $X(w)$ is covered by lines.
Indeed, we may choose $\beta \in S$ such that $s_{\beta} w < w$
for the Bruhat order in $W$; equivalently, $w^{-1}(\beta) \in R^-$.
Then $X(w)$ is stable by the minimal parabolic subgroup 
$P_{\beta}$; moreover, we have 
$C_{w, - \beta} \subset P_{\beta} w x \subset X(w)_{\sm}$ and 
hence the $B$-translates of $C_{w, - \beta}$ cover $X(w)_{\sm}$. 

Given $\beta$ as above, we have 
$w^{-1}(\beta) \in R^- \setminus R_L$. In view of Remark
\ref{rem:simplylaced}, it follows that there exists $v \in W_L$ 
such that $w^{-1}(\beta) = - v(\alpha)$; equivalently, 
$wv(\alpha) = - \beta$. Then $C_{w, - \beta} = wv X(s_{\alpha})$; 
this realizes $C_{w, - \beta}$ as a translate of the Schubert line.
\end{remark}

\begin{example}\label{ex:linesbis}
We illustrate the results of this subsection in the case where
$G := \SL(4)$  (with simple roots $\alpha_1,\alpha_2,\alpha_3$),
and $I := \{ \alpha_1, \alpha_3 \}$; then the parabolic subgroup
$P = P_I = P^{\alpha_2}$ is minuscule.  Let $w := s_1 s_3 s_2$; 
then $w \in W^I$ and $\ell(w) = 3$. The $T$-stable lines through 
$wx$ in $X(w)$ are exactly 
\[
C_1 := C_{w, - \alpha_1} , \quad
C_2 :=  C_{w, - \alpha_1 - \alpha_2 - \alpha_3},  
\quad C_3 := C_{w, -\alpha_3}. \] 
Moreover, 
$C_1 = G_{\alpha_1} wx$, $C_2 = s_1 s_2 G_{\alpha_3} wx$ and 
$C_3 = G_{\alpha_3} wx$. Since $s_1 w, s_3 w < w$, we 
see that $C_1,C_3$ are contained in the smooth locus of $X(w)$. 
But $C_2$ contains $x$, which is the unique singular point 
of $X(w)$.  

The family $\cL_x$ of lines in $G/P$ through $x$ satisfies 
$\cL_x = L X(s_2) \simeq \bP^1 \times \bP^1$;
this isomorphism identifies $X(s_2)$ with
$(\infty,\infty)$. Moreover, the Chow variety of lines in 
$w^{-1} X(w)$ through $x$ is identified with 
$(\bP^1 \times \{ \infty \}) \cup (\{ \infty \} \times \bP^1)$,
the union of two lines meeting at the point $X(s_2) = w^{-1} C_2$. 
These lines are the $B_L$-orbit closures of 
$w^{-1} C_1, w^{-1} C_3$. As a consequence, there are
exactly two families of minimal rational curves on $X(w)$;
those through $wx$ are the $w$-translates of the above lines. 

These results can also be obtained by direct geometric arguments,
since $X = \bG(2,4)$ is embedded in 
$\bP(V(\varpi_2)) = \bP(\Lambda^2 k^4) \simeq \bP^5$ 
as a quadric; moreover, $X(w)$ is  the intersection of $X$ 
with a tangent hyperplane. Thus, $X(w)$ is the projective cone 
over $\bQ^2 \simeq  \bP^1 \times \bP^1$ with vertex $x$. This
cone contains two families of planes (the projective cones
over the two families of lines in $\bP^1 \times \bP^1$)
and the lines in these planes form the two minimal families.
\end{example}

\section{Generalized Bott-Samelson varieties}
\label{sec:perrin}

\subsection{Bott-Samelson desingularizations}
\label{subsec:bott}

We keep the notation of \S\S \ref{subsec:flag} and 
\ref{subsec:schubert}. Let $w \in W^I$. If $w \neq 1$, 
then there exists a decomposition $w = s_{i_1} w'$, 
where $s_{i_1}$ is a simple reflection, $w' \in W$, 
and $\ell(w) =\ell(w') + 1$. It follows that the minimal
parabolic subgroup $P_{i_1}$ stabilizes $X(w)$, and 
$w' \in W^I$. Consider the Schubert variety 
$X(w') \subset G/P$ and the associated fiber bundle 
$P_{i_1} \times^B X(w')$. This is a projective variety 
equipped with an action of $P_{i_1}$ and
an equivariant morphism 
\[ f_{i_1,w'} : P_{i_1} \times^B X(w') \longrightarrow 
P_{i_1}/B \simeq \bP^1, \]
which is a locally trivial fibration (for the Zariski topology)
with fiber $X(w')$. We also have a $P_{i_1}$-equivariant 
morphism 
\[ \pi_{i_1,w'} : P_{i_1} \times^B X(w') \longrightarrow X(w) \] 
which restricts to an isomorphism above the open orbit $Bwx$.

The above ``one-step construction'' can be iterated: 
given a reduced decomposition
\[ \tw = (s_{i_1}, s_{i_2}, \ldots, s_{i_{\ell}}) \]
(i.e., a sequence of simple reflections such that 
$w = s_{i_1} s_{i_2} \cdots s_{i_{\ell}}$ and 
$\ell(w) = \ell$), this yields a projective variety
\[ \tX(\tw) = P_{i_1} \times^B P_{i_2} \times^B
\cdots \times^B P_{i_{\ell}}/B \]
of dimension $\ell$, equipped with an action 
of $P_{i_1}$ and two equivariant morphisms:
\[ f : \tX(\tw) \longrightarrow P_{i_1}/B, \]
a locally trivial fibration with fiber $\tX(\tw')$, 
where $\tw':= (s_{i_2}, \ldots, s_{i_{\ell}})$, and
\[ \pi : \tX(\tw) \longrightarrow X(w), \]
which restricts to an isomorphism above the open orbit $Bwx$. 
Also, note that $\tX(\tw)$ is smooth, and hence $\pi$ 
is a desingularization of $X(w)$. 

More generally, for $1 \leq j \leq \ell$, we have a fibration
\[ f_j : \tX(\tw) \longrightarrow \tX(s_{i_1}, \ldots, s_{i_j}) \]
with fiber $\tX(s_{i_{j+1}}, \ldots, s_{i_{\ell}})$. In particular,
$f_{\ell - 1}$ is a $\bP^1$-bundle; also, $f_1 = f$.

The Bott-Samelson variety $\tX(\tw)$ comes with a base point 
$\tx$, the image of 
$(\dot s_{i_1}, \dot s_{i_2}, \ldots, \dot s_{i_{\ell}})
\in P_{i_1} \times P_{i_2} \times \cdots \times P_{i_{\ell}}$.
Moreover, $\pi(\tx) = wx$ and $B_{\tx} = B_{wx}$.
In particular, $\tx$ is fixed by $T$. Denoting by $\tx_j$
the base point of $\tX(s_{i_1}, \ldots, s_{i_j})$, we have
$f_j(\tx) = \tx_j$. Further, the fiber of $f_j$ at $\tx_j$ is
$T$-equivariantly isomorphic to 
$\tX(s_{i_{j+1}}, \ldots,s_{i_{\ell}})$ on which the $T$-action
is twisted by the Weyl group element $s_{i_1} \cdots s_{i_j}$.

We now recall the description of line bundles on $\tX(\tw)$
obtained in \cite{LT}. For $1 \leq j \leq \ell$, consider the
natural morphism
\[ \pi_j : \tX(s_{i_1}, \ldots, s_{i_j}) \longrightarrow G/B \]
with image the Schubert variety $X(s_{i_1} \cdots s_{i_j})$,
and let 
\[ \cL_j := f_j^* \pi_j^* \cL_{G/B}(\varpi_{i_j}). \]
Then the isomorphism classes of $\cL_1,\ldots, \cL_{\ell}$
form a basis of $\Pic(\tX(\tw))$. Further, for any integers
$n_1, \ldots, n_{\ell}$, the line bundle
$\cL_1^{\otimes n_1} \otimes \cdots \otimes \cL_{\ell}^{\otimes n_{\ell}}$
is ample if and only if $n_1,\ldots,n_{\ell} > 0$; also, every
ample line bundle on $\tX(\tw)$ is very ample. In particular,
$\cL_1 \otimes \cdots \otimes \cL_{\ell}$ is the smallest very
ample line bundle on $\tX(\tw)$. 

Next, we describe the $T$-stable curves through $\tx$ in 
$\tX(\tw)$. Let 
\[ 
\beta_1 := \alpha_{i_1}, \beta_2 := s_{i_1}(\alpha_{i_2}),
\ldots, \beta_{\ell} := s_{i_1} \cdots s_{i_{\ell -1}}(\alpha_{i_{\ell}}). 
\]
Then we have
\begin{equation}\label{eqn:roots} 
R^+ \cap w(R^-) = \{ \beta_1, \beta_2, \ldots, \beta_{\ell} \}. 
\end{equation}
We may now state a version of Lemma \ref{lem:curves} 
for Bott-Samelson varieties:

\begin{lemma}\label{lem:bott}
Keep the above notation.

\begin{enumerate}
 
\item[{\rm (i)}] The $T$-stable curves in $\tX(\tw)$ through $\tx$ 
are exactly the $\tC_j := \overline{U_{\beta_j} \tx}$,
where $1 \leq j \leq \ell$. 

\item[{\rm (ii)}] $\pi$ restricts to isomorphisms
$\tC_j \to C_{w,-\beta_j}$ for all such $j$.

\item[{\rm (iii)}] For $1 \leq j,k \leq \ell$, we have 
$\cL_k \cdot \tC_j = 0$ if $j > k$. If $j \leq k$ then 
$\cL_k \cdot \tC_j$ is the coefficient of 
$\alpha_{i_k}^{\vee}$ in 
$s_{i_k} \cdots s_{i_{j+1}}(\alpha_{i_j}^{\vee})$,
viewed as a linear combination of simple coroots.

\item[{\rm (iv)}] We have
$-K_{\tX(\tw)} \cdot \tC_j = 
\he( s_{i_{\ell}} \cdots s_{i_{j + 1}}(\alpha_{i_j}^{\vee} ) )+ 1$.

\end{enumerate}

\end{lemma}

\begin{proof}
(i) This follows from Lemmas \ref{lem:curves} and
\ref{lem:curvesbis}, since $\pi$ sends $B \tx$ (an open $T$-stable
neighborhood of $\tx$ in $\tX(\tw)$) isomorphically to $Bwx$.

For (ii), note that $\pi$ restricts to a birational morphism
$\tC_j \to C_{w,-\beta_j} \simeq \bP^1$.

(iii) If $j > k$ then $\tC_j$ is contracted by $f_k$. This implies
the first assertion by using the projection formula.

If $j \leq k$ then $f_k$ restricts to an isomorphism of $\tC_j$
onto the $j$th $T$-stable curve in $\tX(s_{i_1},\ldots,s_{i_k})$
through $\tx_k$. Further, the latter curve is sent isomorphically 
by $\pi_k$ to the $T$-stable curve 
$C_{s_{i_1} \cdots s_{i_k},-\beta_j} \subset G/B$.
Using Lemma \ref{lem:curves}, it follows that
\[ \cL_k \cdot \tC_j =
\langle \varpi_{i_k}, - s_{i_k} \cdots s_{i_1}(\beta_j^{\vee}) \rangle 
= \langle \varpi_{i_k}, 
s_{i_k} \cdots s_{i_{j+1}}(\alpha_j^{\vee}) \rangle . \]
This yields the second assertion.

(iv) We first determine $-K_{\tX(\tw)} \cdot \tC_1$. Note that
$G_{\alpha_{i_1}} \subset P_{i_1}$ acts on $\tX(\tw)$ and we have
$\tC_1 = G_{\alpha_{i_1}} \tx$; also, the tangent space of $\tX(\tw)$ 
at $\tx$ is a direct sum of $T$-stable lines with weights 
$\beta_1, \ldots, \beta_{\ell}$. Arguing as in
the proof of Lemma \ref{lem:curves} (iv), it follows that  
\[ 
- K_{\tX(\tw)} \cdot \tC_1 = 
\langle \beta_1 + \cdots + \beta_{\ell}, \alpha_{i_1}^{\vee} \rangle. 
\]
But $\beta_1 + \cdots + \beta_{\ell} = \rho - w(\rho)$
in view of (\ref{eqn:roots}), and hence
\[ - K_{\tX(\tw)} \cdot \tC_1  = 
\langle \rho - w(\rho), \alpha_{i_1}^{\vee} \rangle  =
1 + \he( -w^{-1}(\alpha_{i_1}^{\vee}) ). \]
This yields the assertion, since   
$- w^{-1}(\alpha_{i_1}) =  s_{i_{\ell}} \cdots s_{i_2}(\alpha_{i_1})$.

Next, we determine $-K_{\tX(\tw)} \cdot \tC_j$, where $j \geq 2$.
Then $\tC_j$ is contracted by $f$, i.e., 
$\tC_j \subset F := f^{-1} f(\tx)$. Since $f$ is a locally trivial
fibration, it follows that 
$-K_{\tX(\tw)} \cdot \tC_j = -K_F \cdot \tC_j$. Recall that $F$
is $T$-equivariantly isomorphic to the Bott-Samelson 
variety $\tX(\tw')$ on which the $T$-action is twisted by 
$s_{i_1}$, and this isomorphism sends $\tx$ to the base point
of $\tX(\tw')$. Using an easy induction argument, this 
completes the proof.
\end{proof}

\begin{remark}\label{rem:linesbott}
By Lemma \ref{lem:bott}, we have for $1 \leq k \leq \ell$:
\[ \cL_k \cdot \tC_{\ell} = \begin{cases}
1 & {\rm if}~k = \ell, \\
0 & {\rm else}. \\
\end{cases} \]
Thus, $\tC_{\ell}$ is a line in the smallest projective embedding 
of $\tX(\tw)$.

Also, $\cL_j \cdot \tC_j = 1$ for $1 \leq j \leq \ell$; as a consequence,
$\tC_j$ is a line if and only if $\cL_k \cdot \tC_j = 0$ for all
$j < k$. By an easy argument, this is equivalent to the assertion
that $s_{i_j}$ commutes with $s_{i_{j+1}},\ldots, s_{i_{\ell}}$.
Then there is an isomorphism of resolutions of $X(w)$
\[ \tX(\tw) = \tX(s_{i_1}, \ldots, s_{i_{\ell}}) \simeq
\tX(s_{i_1}, \ldots, s_{i_{j-1}}, s_{i_{j+1}}, \ldots, s_{i_{\ell}}, s_{i_j}) \]
which identifies $\tC_j$ with the line in the right-hand side constructed
as above. 
\end{remark}

Next, we determine the minimal rational curves in $\tX(\tw)$
(these include of course the lines discussed above):

\begin{theorem}\label{thm:bott}
Every minimal family $\cK$ on $\tX(\tw)$ satisfies
$\cK_{\tx} = \{ \tC_j \}$ for some $1 \leq j \leq \ell$. Moreover, 
the minimal rational curves in $\tX(\tw)$ through $\tx$ are exactly 
those $\tC_j$ such that the root 
$s_{i_{\ell}} \cdots s_{i_{j + 1}}(\alpha_{i_j})$ is simple.
\end{theorem}

\begin{proof}
We argue as in the proof of Lemma \ref{lem:lines}. 
By Lemma \ref{lem:covering}, $\cK_{\tx}$ is a projective 
variety; moreover, $\cK$ consists of free curves in view of 
Lemma \ref{lem:smooth}. Since $\cK_{\tx}$ is equipped with 
an action of $T$, it contains a $T$-fixed point, say $\tC_j$.

If $j > 1$ then $\tC_j$ is contracted by $f$. In view of Lemma
\ref{lem:contraction}, it follows that $\cK_{\tx} = \cL_{\tx}$ 
for a unique minimal family $\cL$ on the fiber of $f$ at $\tx$.
Since this fiber is a translate of a smaller Bott-Samelson variety,
we may conclude by induction on $\ell$.

Thus, we may assume that $j = 1$; then $\tC_1$ is the
unique $T$-fixed point of $\cK_{\tx}$. Also, $\cK_{\tx}$ admits
an ample $T$-linearized line bundle: indeed, it is equipped
with a finite $T$-equivariant morphism to some Chow variety 
of $\tX(\tw)$, which in turn is equipped with a finite $T$-equivariant 
morphism to the projectivization of a $T$-module in view of
its construction in \cite[I.3]{Kollar}. As a consequence, $\cK_{\tx}$ 
admits a $T$-equivariant immersion in the projectivization of 
a $T$-module. Using \cite[Prop.~13.5]{Borel}, it follows that 
$\cK_{\tx}$ consists of the unique curve $\tC_1$. 

On the other hand, by Lemma \ref{lem:smooth}, $\tC_1$ 
lies in a unique family of rational curves $\cL$ on $\tX(\tw)$; 
moreover, $\cL$ is covering and satisfies 
\[ \dim(\cL_{\tx}) = - K_{\tX(\tw)} \cdot \tC_1 - 2. \]
In view of Lemma \ref{lem:bott} (iii), this vanishes if and only 
if the root $s_{i_{\ell}} \cdots s_{i_2}(\alpha_{i_1})$ 
is simple.
\end{proof}

\begin{remarks}\label{rem:bott}
(i) Since 
$s_{i_{\ell}} \cdots s_{i_{j + 1}}(\alpha_{i_j}) =- w^{-1}(\beta_j)$,
the simple roots $\alpha_k$ obtained as 
$s_{i_{\ell}} \cdots s_{i_{j + 1}}(\alpha_{i_j})$ for some $j$ 
are exactly those such that $w(\alpha_k) \in R^-$. Then
$\pi$ sends $\tC_j$ to $w X(s_k)$, a translate of
a Schubert line. This relates the minimal rational curves 
in $\tX(\tw)$ through $\tx$ to the  lines in $X(w)$ through 
$wx$ constructed in Lemma \ref{lem:covered}. 
In particular, we may take $j = \ell$, ie.,
$\alpha_k = \alpha_{i_{\ell}}$; this just gives back
the line $\tC_{\ell}$ (Remark \ref{rem:linesbott}).

If $P = P^{\alpha}$ is maximal, then we must have 
$\alpha = \alpha_k$ and $j = \ell$. Thus, $\tX(\tw)$
has a unique minimal family, consisting of the fibers
of the $\bP^1$-bundle 
$f_{\ell - 1} : \tX(\tw) \to \tX(s_{i_1}, \ldots, s_{i_{\ell - 1}})$. 

\smallskip \noindent
(ii) The condition that 
$s_{i_{\ell}} \cdots s_{i_{j + 1}}(\alpha_{i_j})$ is a simple
root, say $\alpha_k$, turns out to be equivalent to the 
exchange condition
\[ s_{i_1} s_{i_2} \cdots s_{i_{\ell}} = 
s_{i_1} \cdots \widehat{s_{i_j}} \cdots s_{i_{\ell}} s_k, \]
where both sides are reduced decompositions of $w$.
\end{remarks}

\subsection{Their generalizations \`a la Perrin}
\label{subsec:perrin}

Let $w \in W$. Recall that the set of simple roots $\alpha$ such 
that $s_{\alpha}$ occurs in a reduced decomposition of $w$ 
is independent of the reduced decomposition, and called the 
\emph{support} of $w$. We denote this set by $\Supp(w)$.
The subgroup of $G$ generated by the $U_{\pm \alpha}$,
where $\alpha \in \Supp(w)$, will be denoted by $G_w$; this is
the derived subgroup of the Levi subgroup $L_{\Supp(w)}$, and
hence is a semi-simple subgroup of $G$, normalized by $T$
and containing a representative of $w$.

Denote by $P^w$ be the largest parabolic subgroup of $G$
such that $P^w \supset B$ and $w \in W^{P^w}$; then 
$P^w = P_{I^w}$, where 
$I^w := \{ \alpha \in S ~\vert~ w(\alpha) \in R^+ \}$. 
Consider the associated Schubert variety $X(w) \subset G/P^w$,
and denote by $P_w$ the closed reduced subgroup of $G$ 
consisting of those $g$ such that $ g X(w) = X(w)$. 
Then $P_w$ is a parabolic subgroup of $G$ containing $B$,
and hence $P_w = P_{I_w}$, where 
$I_w := \{ \alpha \in S ~\vert~ s_{\alpha} w \leq w \}$; here 
$\leq$ denotes the Bruhat order on $W^{I^w}= W/W_{I^w}$. 
Note that $P_w \cap G_w$ is a parabolic subgroup of $G_w$,
and we have
\begin{equation}\label{eqn:schubert}
X(w) = \overline{P_w w P^w}/P^w \simeq 
\overline{(P_w \cap G_w) w (P^w \cap G_w)}/(P^w \cap G_w)
\subset G_w/(P^w \cap G_w).
\end{equation}
We say that $w \in W$ is minuscule if so is $G/P^w$; then 
one may readily check that $G_w/(P^w \cap G_w)$ is minuscule 
as well. Also, note that $G_w$ is simply-laced if so is $G$.
For any minuscule $w \in W$, we have 
$X(w)_{\sm} = (P_w \cap G_w) w x$ by \cite[Prop.~3.3]{BP}. 
In particular, $X(w)$ is smooth if and only if it is homogeneous 
under $P_w \cap G_w$.

Next, let $w = w_1 w'$, where $w_1, w' \in W$ satisfy
$P^{w_1} \cap G_{w_1} \subset P_{w'}$; equivalently, we have
$I^{w_1} \cap \Supp(w_1) \subset I_{w'}$. We may then define
\[ \tX(w_1,w') := 
\overline{(P_{w_1} \cap G_{w_1}) w_1 (P^{w_1} \cap G_{w_1})}
\times^{P^{w_1} \cap G_{w_1}} X(w'). \]
This is a projective variety equipped with an action of
$P_{w_1} \cap G_{w_1}$ and an equivariant morphism
\[ f_{w_1,w'} : \tX(w_1,w') \longrightarrow X(w_1), \]
which is a Zariski locally trivial fibration with fiber $X(w')$.
If in addition $\ell(w) = \ell(w_1) + \ell(w')$, then
$w' \in W^{P^w}$ and hence $P^{w'} \supset P^w$.
Thus, if $P^w$ is maximal and $w' \neq 1$, then $P^{w'} = P^w$. 
Under these assumptions, we obtain another equivariant morphism
\[ \pi_{w_1,w'} : \tX(w_1,w') \longrightarrow G/P^w. \]
One may check that $\pi_{w_1,w'}$ is birational to its image
$X(w)$; it restricts to an isomorphism above the open orbit 
$Bwx$. Also, note that 
\begin{equation}\label{eqn:sub} 
P_{w_1} \cap G_{w_1} \subset P_w \cap G_w. 
\end{equation}

\begin{remark}\label{rem:perrin}
If $w_1$ is a simple reflection $s_{\alpha}$, then $P^{w_1}$ 
is the maximal parabolic subgroup 
$P^{\alpha} = P_{S \setminus \{ \alpha \} }$,
and $X(w_1)$ is the Schubert line in $G/P^{\alpha}$. 
Moreover, $G_{w_1} =G_{\alpha}$ and 
$P_{w_1} = P_{\{ \alpha \} \cup \alpha^{\perp} }$.
So $P^{w_1} \cap G_{w_1} = B \cap G_{\alpha}$ 
is a Borel subgroup of $G_{\alpha}$. Thus, we have
\[ \tX(s_{\alpha},w') = 
G_{\alpha} \times^{B \cap G_{\alpha}} X(w') \simeq
P_{\alpha} \times^B X(w'), \]
with fibration $f_{s_{\alpha},w'}$ over 
$P_{\alpha}/B \simeq \bP^1$. If in addition 
$\ell(w) = \ell(w') + 1$, then $\pi_{s_{\alpha},w'}$ yields
a birational morphism to $X(w)$. Thus, the above ``one-step 
construction'' generalizes that of Bott-Samelson varieties.
\end{remark}

This construction can be iterated, under certain additional
assumptions that are discussed in detail in \cite[\S 5.2]{Pe07}.
We now present some notions and results from [loc.~cit.]:
a finite sequence $\hw = (w_1,\ldots,w_m)$ of elements of $W$
is called a \emph{generalized reduced decomposition of} $w$,
if we have $w = w_1 \cdots w_m$ and 
$\ell(w) = \ell(w_1) + \cdots + \ell(w_m)$.
Such a decomposition is called \emph{good} if in addition 
$w$ is minuscule and we have 
\[ I^{w_i} \cap \Supp(w_i) \subset I_{w_{i+1} \cdots w_m}
\subset w_i^{\perp} \cup \Supp(w_i) \quad (1 \leq i \leq m -1), \]
where $w_i^{\perp}$ denotes the set of simple roots $\alpha$ 
such that $s_{\alpha}$ commutes with $w_i$. 
Under these assumptions, $(w_{i+1},\ldots,w_m)$ is a good 
generalized reduced decomposition of $w_{i+1} \cdots w_m$ 
for $i = 1, \ldots, m-1$. Moreover, 
$P^{w_i} \cap G_{w_i} \subset P_{w_{i+1} \cdots w_m}$
for all such $i$. 

Given a good generalized reduced decomposition $\hw$ of $w$, 
we obtain a projective variety $\hX(\hw)$ equipped with an action
of $P_w$, a locally trivial fibration
\[ \hf : \hX(\hw) \longrightarrow X(w_1) \]
with fiber $\hX(w_2,\ldots,w_m)$, and a birational morphism
\[ \hpi : \hX(\hw) \longrightarrow X(w). \]
Also, $\hX(\hw)$ has a base point $\hx$ such that
$\hpi(\hx) = wx$ and $\hf(\hx) = w_1 x_1$, where $x_1$ denotes
the base point of $G/P^{w_1}$. 

More generally, for $1 \leq i \leq m$, we have a fibration 
\[ \hf_i : \hX(\hw) \longrightarrow \hX(w_1,\ldots,w_i) \]
with fiber $\hX(w_{i+1},\ldots,w_m)$. Further, $\hf_i(\hx) = \hx_i$ 
with an obvious notation, and $\hf_1 = \hf$.

By \cite[\S 5.1]{Pe07}, the morphism $\hf$ is $P_w$-equivariant. 
As a consequence, we have the equality of stabilizers 
$P_{w,\hx} = P_{w,wx} = P_w \cap w P^w w^{-1}$. 
Thus, $P_{w,\hx}$ contains the maximal torus $T$.
Also, note that $P_{w,\hx}$ is smooth and connected,
in view of the following result:

\begin{lemma}\label{lem:intersection}
Let $P, Q$ be two parabolic subgroups of $G$ containing the
maximal torus $T$. 

\begin{enumerate}

\item[{\rm (i)}] The (scheme-theoretic) intersection $P \cap Q$
is smooth and connected.

\item[{\rm (ii)}] Denote by $L$ (resp.~$M$) the Levi subgroup
of $P$ (resp.~$Q$) containing $T$. Then $P \cap Q$ has a 
Levi decomposition with Levi subgroup $L \cap M$.

\end{enumerate}

\end{lemma}

\begin{proof}
(i) The smoothness of $P \cap Q$ follows from 
\cite[13.21]{Borel}, and the connectedness from [loc.~cit., 14.22].

(ii) This is a consequence of \cite[Prop.~2.1]{DM}.
\end{proof}

\subsection{Structure of minimal families}
\label{subsec:structure}

We still consider a good generalized reduced decomposition  
$\hw = (w_1, \ldots,w_m)$ of a minuscule element $w \in W$, 
and set $P := P^w$. Also, we choose reduced decompositions
\[ \tw_i = (s_{i,1}, \ldots, s_{i,\ell_i}) \]
of $w_i$ for $1 \leq i \leq m$. This yields a reduced
decomposition of $w$ by concatenation, and hence
a Bott-Samelson variety $\tX(\tw)$. By using \cite[\S 5.3]{Pe07},
we obtain a commutative diagram of pointed varieties
\[ \xymatrix{
(\tX(\tw),\tx) \ar[r]^{\tpi} \ar[d]_{\tf}  &
(\hX(\hw),\hx) \ar[r]^{\hpi}  \ar[d]^{\hf} &
(X(w),wx) \\
(\tX(\tw_1),\tx_1) \ar[r]^-{\tpi_1} & (X(w_1),w_1 x_1), \\
}\]
where the horizontal arrows induce local isomorphisms 
at the corresponding base points, and the
vertical arrows are locally trivial fibrations; moreover,
the composition $\hpi \circ \tpi$ is the Bott-Samelson
resolution $\pi : \tX(\tw) \to X(w)$.

By arguing as in the proof of Lemma \ref{lem:bott},
one checks that $\tpi$ and $\hpi$ induce 
isomorphisms on $T$-stable curves through the respective
base points. Thus, we may index the $T$-stable curves
through $\hx$ in $\hX(\hw)$ as $\hC_{i,j}$, where
$1 \leq i \leq \ell_j$ and $1 \leq j \leq m$. 
Note that $\hf$ contracts all the $\hC_{i,j}$ with $j \geq 2$,
and sends each $\hC_{i,1}$ isomorphically to  
$C_{w_1,-\beta_i}$; in particular, $\hf$ yields a bijection
from $\{ \hC_{1,1}, \ldots, \hC_{\ell_1,1} \}$ to the set of
$T$-stable curves through $w_1 x_1$ in $X(w_1)$.
Also, every minimal family on $\hX(\hw)$ contains some 
$\hC_{i,j}$, as follows from Borel's fixed point theorem.

We now assume that $\hX(\hw)$ is smooth; equivalenty,
$X(w_i)$ is smooth for $i = 1, \ldots, m$. Then each
$\hC_{i,j}$ is an embedded free rational curve (Lemma 
\ref{lem:smooth}). Let $\cK = \cK_{i,j}$ be the family of 
rational curves on $\hX(\hw)$ that contains $\hC_{i,j}$; 
then $\cK$ is covering in view of Lemma \ref{lem:covering}. 
If $j \geq 2$ then by Lemma \ref{lem:contraction}, 
there exists a unique covering family $\cL$ of rational curves 
on $\hY := \hX(w_2,\ldots, w_m)$ such that 
$\cK_{\hx} = \cL_{\hy}$, where $\hy$ denotes the
base point of $\hY$; the above isomorphism is 
$T$-equivariant, where the $T$-action on $\hY$
is twisted by $w_1$. Arguing by induction on $m$, 
we may thus reduce to the case where $j = 1$.

Assume in addition that $G$ is simply-laced and
$w_1,\ldots,w_m$ are minuscule; then each $X(w_i)$ 
is a smooth Schubert variety in the minuscule homogeneous
space $G/P^{w_i}$, and hence is a minuscule homogeneous
space as well (see \cite[Prop.~3.3]{BP}). By combining 
Lemma \ref{lem:rational}, Remark \ref{rem:rational} 
and Proposition \ref{prop:long}, we obtain 
two $P_{w,\hx}$-equivariant rational maps 
\begin{equation}\label{eqn:rational}
\hpi_* : \cK_{\hx} \dasharrow \cL(w)_{wx}, \quad 
\hf_* : \cK_{\hx} \dasharrow \cL(w_1)_{w_1 x_1}, 
\end{equation}
where $\cL(w)$ is a family of lines in $X(w)$, and $\cL(w_1)$
the family of all lines in $X(w_1)$. 

We may now obtain a qualitative analogue of the description 
of minimal families in minuscule varieties (Proposition 
\ref{prop:long} and Remark \ref{rem:simplylaced}):

\begin{proposition}\label{prop:isom}
Assume that $G$ is simply-laced. Let $w \in W$ be a minuscule
element, and $\hw = (w_1,\ldots,w_m)$ a good generalized reduced 
decomposition of $w$, where $w_1,\ldots,w_m$ are minuscule and 
$X(w_1),\ldots,X(w_m)$ are smooth. Let $\cK$ be a minimal family 
on $\hX(\hw)$.

\begin{enumerate}

\item[{\rm (i)}] $\cK_{\hx}$ consists of embedded free curves.

\item[{\rm (ii)}] The rational maps (\ref{eqn:rational}) are immersions.

\item[{\rm (iii)}] $\cK_{\hx}$ is a minuscule homogeneous space.

\item[{\rm (iv)}] The tangent map (\ref{eqn:tangent})
yields an immersion of $\cK_{\hx}$ into $\bP(T_{\hx} \hX(\hw))$.

\end{enumerate}

\end{proposition}

\begin{proof}
(i) By Lemma \ref{lem:smooth}, every curve in $\cK_{\hx}$ is free. 
Also, recall that being embedded is an open property, invariant
under the $T$-action. Since every $T$-fixed curve in $\cK_{\hx}$
is embedded, this yields the assertion.

(ii)  In view of Lemma \ref{lem:rational}, $\hf_*$ 
is defined at any $T$-fixed point, and hence everywhere 
by the above argument. Also, $\hC_{i,1}$ is the unique 
$T$-fixed point of its fiber under $\hf_*$. By arguing as in 
the proof of Theorem \ref{thm:bott}, it follows that this fiber 
consists of a unique point. Thus, all the fibers of $\hf_*$ are 
finite by upper semi-continuity of the dimensions of fibers. 
So $\hf_*$ is a finite morphism.

Viewing $\hpi$ as a morphism to $G/P$ and adapting 
the arguments of the above paragraph, we see that
 $\hpi_*$ is a finite morphism as well.
Its image is contained in the variety $\cL_{wx}$ of lines
in $G/P$ through $wx$, which is a minuscule homogeneous space 
under $w L w^{-1}$ (Remark \ref{rem:simplylaced}).
We now use the equivariance of $\hpi_*$ under
$P_{w,\hx} = P_{w,wx} = P_w \cap w P w^{-1}$, and hence
under $w B_L w^{-1}$, a Borel subgroup of $w L w^{-1}$.
The image of $\hpi_*$ is a Schubert subvariety of $\cL_{wx}$
with respect to this Borel subgroup. By arguing as in the proof
of Proposition \ref{prop:linesbis}, it follows that $\hpi_*$ is 
an immersion. Likewise, $\hf_*$ is an immersion as well.

(iii) Consider the universal family $\rho: \cU \to \cK$. Then
$\cU_{\hx}$ is smooth by (i) and Lemma \ref{lem:smooth}.
Using (i) again and Lemma \ref{lem:em}, it follows that 
$\cK_{\hx}$ is smooth as well. So $\cK_{\hx}$ is isomorphic 
to a smooth Schubert variety in the minuscule homogeneous 
space $\cL_{wx}$. This implies the statement in view of 
\cite[Prop.~3.3]{BP}.

(iv) By (i) and (\ref{eqn:tangentbis}), we have 
a commutative diagram
\[ \xymatrix{
\cK_{\hx} \ar[r] \ar[d]_{\hpi_*} & \bP(T_{\hx} \hX(\hw)) 
\ar[d]^{d\hpi_{\hx}} \\
\cL(w)_{wx} \ar[r] & \bP(T_{wx} X(w)), \\
} \]
where the horizontal arrows are the tangent maps.
Moreover, $\hpi_*$ is an immersion,
$d\hpi_{\hx}$ is an isomorphism, and the bottom
horizontal arrow is an immersion as well by Proposition 
\ref{prop:long}. This yields the assertion.
\end{proof}

Next, we obtain a more quantitative version of Proposition
\ref{prop:isom} under additional assumptions. We will need
the following observation:

\begin{lemma}\label{lem:inc}
If $X(w)$ is smooth, then $G_w \subset P_w$.
\end{lemma}

\begin{proof}
By \cite[Prop.~3.3]{BP} and the smoothness assumption, 
$X(w)$ is a unique orbit of $P_w$. Since $x \in X(w)$,
it follows that $X(w) = P_w x$. Thus, we have $w = u v$ 
for some $u \in W_{I_w}$ and $v \in W_{I^w}$. Using a reduced 
decomposition of $w$ and 
the fact that $w \in W^{I^w}$, it follows that $w \in W_{I_w}$. 
Therefore, $\Supp(w) \subset I_w$. This completes the proof.
\end{proof}

Since $X(w_1)$ is smooth, we have $G_{w_1} \subset P_{w_1}$ 
by the above lemma. In view of (\ref{eqn:sub}), it follows that
$G_{w_1}$ is a subgroup of $P_w \cap G_w$; also, 
$\hf : \hX(\hw) \to X(w_1)$ is clearly equivariant under this subgroup, 
and sends $\hx$ to $w_1 x_1$. Thus, we have the inclusion of stabilizers
\begin{equation}\label{eqn:incl}
G_{w_1,\hx} \subset G_{w_1,w_1 x_1}.
\end{equation}
Moreover, 
$G_{w_1,\hx} = G_{w_1} \cap P^w_{wx} = G_{w_1} \cap w P^w w^{-1}$ 
is the intersection of two parabolic subgroups of $G_{w_1}$;
both contain $T_1 := T \cap G_{w_1}$ as a maximal torus.
By Lemma \ref{lem:intersection}, it follows that 
$G_{w_1,\hx}$ is smooth, connected,
and admits a Levi subgroup containing $T_1$; all these properties 
also hold for $G_{w_1,w_1 x_1}$. We may now state our result:

\begin{proposition}\label{prop:final}
Let $G$ be a simply-laced semi-simple algebraic group,
$w \in W$ a minuscule element, and $\hw =(w_1,\ldots,w_m)$ 
a good generalized reduced decomposition of $w$.
Assume that $w_1,\ldots,w_m$ are minuscule, 
$X(w_1), \ldots, X(w_m)$ are smooth, and the inclusion
(\ref{eqn:incl}) induces an equality of Levi subgroups
containing $T_1$.
Let $\cK$ be a family of rational curves on $\hX(\hw)$
containing a $T$-stable curve $\hC$ which is not contracted 
by $\hf$. Then every $\alpha \in S$ such that
$w_1^{-1}(\alpha) \in R^-$ satisfies $w^{-1}(\alpha) \in R^-$ 
and
\begin{equation}\label{eqn:dimension} 
\he( -w^{-1}(\alpha) ) \geq  \he( -w_1^{-1}(\alpha)). 
\end{equation}
Moreover, equality holds for some $\alpha$ as above 
if and only if $\cK$ is minimal, and then the morphisms
$\hpi_* : \cK_{\hx} \to \cL(w)_{wx}$, 
$\hf_*: \cK_{\hx} \to \cL(w_1)_{w_1 x_1}$ are isomorphisms.
\end{proposition}

\begin{proof}
We may choose a reduced decomposition
$(s_{1,1}, \ldots, s_{1,\ell_1})$ of $w_1$ such that
$\alpha = \alpha_{1,1}$. Then $\hC = \hC_{i,1}$
where $1 \leq i \leq \ell_1$.

Denote by $G_1 \supset T_1$ the common Levi
subgroup to $(P_{w_1} \cap G_{w_1})_{\hx}$ and
$(P_{w_1} \cap G_{w_1})_{w_1 x_1}$. Since $X(w_1)$ 
is a minuscule variety under $P_{w_1} \cap G_{w_1}$, 
it follows from Proposition \ref{prop:long} and Remark 
\ref{rem:simplylaced} (i) that $\cL(w_1)_{w_1x_1}$ is 
a minuscule homogeneous space under $G_1$. Therefore, 
the $T$-stable curves through $w_1 x_1$ in $X(w_1)$ 
form a unique orbit of the Weyl group $W_1$ of 
$(G_1,T_1)$. So the same holds for the $T$-stable curves 
$\hC_{1,1}, \ldots, \hC_{\ell_1,1}$. Since $\cK_{\hx}$ is 
stable under $G_1$, it contains all the latter curves.

As a consequence, we have 
\[ \dim(\cK_{\hx}) = -K_{\hX(\hw)} \cdot \hC_{i,1} - 2
= -K_{\hX(\hw)} \cdot \hC_{1,1} - 2. \]
We now determine $-K_{\hX(\hw)} \cdot \hC_{1,1}$. 
Note that the morphism 
$\tpi: \tX(\tw) \to \hX(\hw)$ restricts to an isomorphism 
over the open orbit of the minimal parabolic subgroup
$P_{\alpha}$ (a subgroup of $P_w$) 
in $\hX(\hw)$; in particular, $\tpi$ restricts to an 
isomorphism from $\tC_1 = G_{\alpha}\tx$ to 
$\hC_{1,1}$. This yields the equality
$-K_{\hX(\hw)} \cdot \hC_{1,1} = -K_{\tX(\tw)} \cdot \tC_1$.
By using Lemma \ref{lem:bott} (iii), this yields in turn
\[ \dim(\cK_{\hx}) = \he( -w^{-1}(\alpha^{\vee}) ) -1. \] 
The rational map $\hf_*$ is $G_1$-equivariant, and hence
dominant as $\cL(w_1)_{w_1 x_1}$ consists of a unique orbit 
of $G_1$. Thus, 
\[ \dim(\cK_{\hx}) \geq \dim(\cL(w_1)_{w_1 x_1}). \] 
But we have 
\[ \dim(\cL(w_1)_{w_1 x_1}) = \he(-w_1^{-1}(\alpha^{\vee})) -1 \]
as follows by arguing as above with the morphism 
$\tpi_1 : \tX(\tw_1) \to X(w_1)$. This proves the inequality 
(\ref{eqn:dimension}). 

If the family $\cK$ is minimal, then $\hf_*$ is an
immersion (Proposition \ref{prop:isom}). It follows 
that equality holds in (\ref{eqn:dimension}) and $\hf_*$ 
is surjective. Also, $\hpi_*$ is an immersion (Proposition
\ref{prop:isom} again), and is birational in view of 
Remark \ref{rem:rational}; thus, $\hpi_*$ is an isomorphism.

Conversely, assume that equality holds in 
(\ref{eqn:dimension}); then the dominant rational map 
$\hf_*$ is generically finite. Using $G_1$-equivariance 
and homogeneity of $\cL(w_1)_{w_1 x_1}$ once more, 
it follows that $\hf_*$ is an isomorphism. In particular, 
$\cK_{\hx}$ is projective, i.e., $\cK$ is minimal.
\end{proof}

The assumptions of the above proposition hold for the 
generalized Bott-Samelson resolutions obtained by
Construction 1 of Perrin (see \cite[\S 5.4]{Pe07}; these yield 
all small resolutions of $X(w)$ in view of [loc.~cit., Cor.~7.9]). 
More specifically, any such decomposition is good and consists 
of minuscule elements by [loc.~cit., 5.4] again. Also, the inclusion 
(\ref{eqn:incl}) induces an equality of Weyl groups relative 
to $T_1$ (see Proposition \ref{prop:weyl}, or
\cite[Thms.~4.1, 4.8, 4.14]{BK}), and hence of Levi subgroups. 
Using Proposition \ref{prop:root}, or \cite[Props.~4.7, 4.13, 4.15]{BK}, 
we now obtain a complete description of minimal families in this setting:

\begin{theorem}\label{thm:final}
Let $G$ be a simply-laced semi-simple algebraic group,
$w \in W$ a minuscule element, and $\hw = (w_1,\ldots,w_m)$ 
a generalized reduced decomposition of $w$ obtained
by Construction 1. Assume that $X(w_1), \ldots, X(w_m)$
are smooth, and consider a family $\cK$ of minimal
rational curves on $\hX(\hw)$. Then there exist
an integer $1 \leq i \leq m$ and an isomorphism 
$\hX(w_i,\ldots,w_m) \simeq Y \times Z$,
where $Y$ is a minuscule homogeneous space, 
such that $\cK_{\hx}$ consists of the lines in
$Y$ through $y$. Here $\hX(w_i, \ldots, w_m)$ is
identified with the fiber at $\hx$ of the fibration 
$\hf_{i-1} : \hX(\hw) \to \hX(w_1,\ldots,w_{i-1})$,
and $y$ denotes the image of $\hx$ under the projection 
to $Y$.
\end{theorem}

\begin{proof}
We argue as in the proofs of Lemma \ref{lem:lines}
and Theorem \ref{thm:bott}. By Borel's fixed point theorem, 
$\cK_{\hx}$ contains a $T$-stable curve, i.e., some curve 
$\hC_{i,j}$. If $j \geq 2$ then $\hC_{i,j}$ is contracted
by $\hf$. Using Lemma \ref{lem:contraction}, it follows that
$\cK_{\hx} = \cL_{\hx}$ for a unique minimal family 
$\cL$ on the fiber of $\hf$ at $\hx$. As this fiber is
isomorphic to $\hX(w_2,\ldots,w_m)$, we conclude by
induction on $m$.

Thus, we may assume that $\cK_{\hx}$ contains a $T$-stable curve 
of the form $\hC_{i,1}$. Also, decomposing $G$ into a direct
product of simple groups yields a decomposition of
$\hX(\hw)$ into a product of the associated 
generalized Bott-Samelson resolutions. So we may further 
assume that $G$ is simple by using Lemma \ref{lem:product}. 

By Lemma \ref{lem:us}, there exists a unique 
$\alpha \in S$ such that $w_1^{-1}(\alpha) \in R^-$.
Using Proposition \ref{prop:final}, we obtain that
$w^{-1}(\alpha) \in R^-$ and 
$\he(- w^{-1}(\alpha)) = \he(-w_1^{-1}(\alpha))$. 
If $w \neq w_1$ then we have
$w_1^{-1}(\alpha) > w^{-1}(\alpha)$
by Proposition \ref{prop:root}, a contradiction. Thus,
$w = w_1$; then $\hX(\hw) = X(w_1)$ is just a minuscule 
homogeneous space, and we conclude by Proposition
\ref{prop:long}.
\end{proof}

\begin{remark}\label{rem:final}
The description of line bundles on Bott-Samelson varieties extends 
to their generalized versions obtained by Construction 1 (see 
\cite[\S 6.1]{Pe07}), and Lemma \ref{lem:bott} can also 
be extended to this setting. In particular, $\tX(\tw)$ admits
a smallest very ample line bundle, and the $T$-stable curves 
$\tC_{i,m}$ are lines in the corresponding projective embedding. 
In fact, all lines are obtained by a variant of this construction, 
as follows from Theorem \ref{thm:final}. 

Also, the canonical class of $\tX(\tw)$ is described in combinatorial 
terms in \cite[\S 6.2]{Pe07}. This yields a formula for the dimension
of the family of rational curves on $\tX(\tw)$ containing a given 
$T$-stable curve. But we do not know how to deduce the above 
theorem directly from this formula.
\end{remark}

\begin{example}\label{ex:final}
As in Example \ref{ex:linesbis}, we illustrate the above results 
in the case where $G = \SL(4)$ and 
$w = s_1 s_3 s_2 = s_3 s_1 s_2$. The Schubert variety $X(w)$ 
admits three generalized Bott-Samelson desingularizations, 
displayed in the following commutative diagram:
\[ \xymatrix{ 
& \tX(s_1,s_3,s_2) \simeq \tX(s_3,s_1,s_2) 
\ar[dl]_{\tpi_{32}}  \ar[dr]^{\tpi_{12}}   &  \\
\hX(s_1, s_3 s_2) \ar[dr]^{\hpi_{32}} \ar[d]^{\hpi_1} &   &
\hX(s_3,s_1s_2) \ar[dl]_{\hpi_{12}} \ar[d]^{\hpi_3}  \\
X(s_1) &  X(w)  &  X(s_3), \\
} \]
where $X(s_1) \simeq \bP^1 \simeq X(s_3)$ and 
$\hpi_1$ (resp.~$\hpi_3$) is a locally trivial fibration with
fiber $X(s_3 s_2)$ (resp.~$X(s_1s_2)$); both fibers are
isomorphic to $\bP^2$. Moreover, the Bott-Samelson variety 
$\tX(s_1,s_3,s_2)$ admits a unique minimal family, consisting 
of the fibers of the natural morphism 
$\tX(s_1,s_3,s_2) \to \tX(s_1,s_3) \simeq \bP^1 \times \bP^1$.
Likewise, $\hX(s_1,s_3 s_2)$ (resp.~$\hX(s_3,s_1 s_2)$)
admits a unique minimal family, consisting of the lines
in the fibers of $\hpi_1$ (resp.~$\hpi_3$). The two latter 
minimal families are sent isomorphically to the two families 
of lines in $X(w)$.

In geometric terms, $\hpi_{32}$ (resp.~$\hpi_{12}$) is the 
blowing-up of $X(w)$ along its Weil divisor $X(s_3 s_2)$
(resp.~$X(s_1s_2)$), and the composition
$\pi = \hpi_{32} \circ \tpi_{32} = \hpi_{12} \circ \tpi_{12}$ 
is the blowing-up at the singular point $x$. The above diagram 
gives back the Atiyah flop (see \cite{Atiyah}).
\end{example}

\section{Some combinatorial results on generalized Bott-Samelson varieties}
\label{sec:comb}

\subsection{A sequence of roots}
\label{subsec:seq}

Throughout this section, we consider a simple, simply-laced 
and simply-connected algebraic group $G$, a minuscule 
parabolic subgroup $P = P_I = P^{\alpha}$, and a Weyl
group element $w \in W^I$.
 
We identify roots and coroots via a $W$-invariant 
scalar product $ \langle \, , \, \rangle$. For any $v \in W$, we set 
$R^+(v) := \{ \gamma \in R^+ ~\vert~ v(\gamma) \in R^- \}$;
then $\vert R^+(v) \vert = \ell(v)$.

Let $\hw = (w_1, w_2, \ldots , w_m)$ be a generalized reduced 
decomposition of $w$ obtained by Construction 1 of \cite[5.4]{Pe05}.
Choose a reduced decomposition
$\tw = (s_{\beta_1}, \ldots, s_{\beta_r})$ of $w$ that refines the above
generalized reduced decomposition. We will freely use the associated
quiver $Q_w$, as defined in \cite[2.1, 4.2]{Pe05}. In particular, 
this quiver has vertices $1, \ldots, r$, colored by simple roots via 
the map $\beta : j \mapsto \beta_j$. The set of peaks
$\Peaks(Q_w)$ is equipped with an ordering 
$i_1 \preceq  i_2 \preceq \ldots \preceq i_m$ 
that defines the above generalized reduced decomposition. 
We illustrate this on two examples which will be reconsidered
repeatedly:

\begin{example}\label{ex:SL5}
Let $G = \SL_5$ (a simple group of type $\rA_4$) and $P=P^{\alpha_{2}}$.
Then $X = G/P$ is the Grassmannian $\bG(2,5)$. Take 
\[ w = s_2s_1s_4s_3s_2 \] 
so that $X(w)$ is the Schubert divisor in $X$. 
We have $\Peaks(Q_w) = \{1,3\}$ and  
$\beta(\Peaks(Q_w)) = \{ \alpha_2, \alpha_4 \}$.
The generalized reduced decomposition associated with the standard
ordering $\alpha_2 \preceq \alpha_4$ of $\Peaks(Q_w)$  has 
$w_1 = s_2s_1$ and $w_2 = s_4s_3s_2$. For the reverse ordering,
$\alpha_4 \preceq' \alpha_2$, we have $w'_1 = s_4$ and 
$w'_2 = s_2s_1s_3s_2$.
\end{example}

\begin{example}\label{ex:SL9} 
Let $G = \SL_9$ (type $\rA_8$) and $P=P^{\alpha_4}$; then
$X = G/P = \bG(4,9)$. Take 
\[ w = s_3s_2s_1s_5s_4s_3s_2s_6s_5s_4s_3s_8s_7s_6s_5s_4, \] 
then $X(w)$ is a Schubert variety of dimension $16$. We have 
$\Peaks(Q_w) = \{ 1, 4, 12 \}$ and 
$\beta(\Peaks(Q_w)) = \{ \alpha_3, \alpha_5, \alpha_8 \}$. 
For the standard ordering $\alpha_3 \preceq \alpha_5 \preceq \alpha_8$
of $\Peaks(Q_w)$, we have $w_1 = s_3s_2s_1$, $w_2 = s_5s_4s_3s_2s_6s_5s_4s_3$, 
and $w_3 = s_8s_7s_6s_5s_4$. 
For the ordering $\alpha_8 \preceq' \alpha_3 \preceq' \alpha_5$, we have  
$w'_1 = s_8$, $w'_2 = s_3s_2s_1$, and $w'_3 = s_5s_4s_3s_2s_6s_5s_4s_3s_7s_6s_5s_4$. 
\end{example}

\begin{lemma}\label{lem:us} 
With the above notation, we have $R^+(w_1^{-1}) \cap S = \{ \beta_1 \}$.
\end{lemma}

\begin{proof}
By \cite[Prop.~5.13]{Pe07}, $Q_{w_1}$ has only one peak, namely, 
$i_1$. On the other hand, we have 
$\beta(\Peaks(Q_{w_1})) = R^+(w_1^{-1}) \cap S$.
This yields the assertion.
\end{proof}

By construction, there exists an increasing sequence 
$l_1 = 1 \leq l_2  < l_3 < \ldots < l_{m+1} = r$ 
of positive integers such that 
$w_1 = s_{\beta_{l_1}}\cdots s_{\beta_{l_2}}$ and
$w_j = s_{\beta_{l_j + 1}}\cdots s_{\beta_{l_{j+1}}}$ 
for all $2 \leq j \leq m$.  
 
Let $v_i := s_{\beta_i}s_{\beta_{i-1}} \cdots s_{\beta_1}$ ($1\leq i \leq r$).
Let $\gamma_i := v_i(\beta_1)$ ($1\leq i \leq r$); then $\gamma_i$ is
a negative root. In the rest of this subsection, we prove the following:

\begin{proposition}\label{prop:leq}
$\gamma_{i+1}\leq \gamma_i$ for all $1\leq i \leq r-1$.
\end{proposition}

We first obtain two preliminary results.

\begin{lemma}\label{lem:-1} 
Let $\gamma \in R^+$. Let $v \in W$ be a minimal element 
such that $v(\gamma) = \alpha_0$, the highest root. Then for any 
$\mu \in R^+(v)$, we have $\langle \mu, \gamma  \rangle =-1$.
\end{lemma}

\begin{proof} 
By induction on $\ell(v)$. If $\ell(v)=1$, then we have $v=s_i$ 
for some $i$ and $\gamma = s_i(\alpha_0) = \alpha_0 - \alpha_i$, 
since $G$ is simply-laced. So, 
$\langle \alpha_i, \gamma  \rangle = -1$. 

Assume that $\ell(v)\geq 2$. Choose an integer $i$ such that 
$\ell(vs_i) = \ell(v) - 1$. Since $v$ is minimal such that 
$v(\gamma) = \alpha_0$, we have 
$\langle \alpha_i, \gamma  \rangle =-1$.
On the other hand by induction, for any 
$\mu \in R^+(v) \setminus \{ \alpha_i \}$, we have 
$\langle \mu, \gamma  \rangle = \langle s_i(\mu), s_i(\gamma)  \rangle = -1$.
\end{proof}

\begin{lemma}\label{lem:geq}
For any $\mu \in R^+(w^{-1})$, we have $\langle \mu, \beta_1  \rangle \geq 0$. 
\end{lemma}

\begin{proof} 
We have $-w^{-1}(\mu) , -w^{-1}(\beta_1) \in R^+(w)$, and hence,  
$\alpha \leq -w^{-1}(\mu)$ and  $\alpha \leq -w^{-1}(\beta_1)$.  Now, if  
$\langle \mu, \beta_1  \rangle \leq -1$, then $\mu + \beta_1$ is a root and 
$\langle \varpi_{\alpha} , -w^{-1}(\mu + \beta_1) \rangle \geq 
\langle \varpi_{\alpha} , 2 \alpha \rangle = 2$, 
contradicting the fact that $\varpi_{\alpha}$ is minuscule.
Thus, we have $\langle \mu, \beta_1  \rangle \geq 0$. 
\end{proof}

We may now prove Proposition \ref{prop:leq}.
Let $v \in W$ be a minimal element such that $v(\beta_1) = \alpha_0$. 
Then by Lemmas \ref{lem:-1} and \ref{lem:geq}, we have 
$R^+(w^{-1}) \cap R^+(v) = \emptyset$. Therefore, we have 
$\ell(w^{-1}v^{-1}) = \ell(w) + \ell(v)$.
In particular, we have $\ell(v_i v^{-1}) = \ell(v) + i$ for all $1\leq i \leq r$. 
Therefore, $(v_i v^{-1})^{-1}(\beta_{i+1})$ is a positive root.   
Since $\alpha_0$ is dominant,  we have  
\[ \langle v_i v^{-1}(\alpha_0) , \beta_{i+1} \rangle
= \langle \alpha_0,  (v_i v^{-1})^{-1}(\beta_{i+1}) \rangle \geq 0. \]  
Therefore, for all $1\leq i \leq r-1$, we have 
\[ \gamma_{i+1}=v_{i+1}(\beta_1)=v_{i+1}v^{-1}(\alpha_0)
=v_{i}v^{-1}(\alpha_0) - \langle v_{i}v^{-1}(\alpha_0) , \beta_{i+1} \rangle \beta_{i+1}
\leq v_{i}v^{-1}(\alpha_0) = \gamma_i. \]
This completes the proof.

\subsection{Root inequality}
\label{subsec:root}

We keep the notation of Subsection \ref{subsec:seq}, and prove
the following:

\begin{proposition}\label{prop:root}
Assume that $w \neq w_1$. Then we have
$w_1^{-1}(\beta_1) > w^{-1}(\beta_1)$.
\end{proposition}

\begin{example}\label{ex:SLbis}
With the notation of Example \ref{ex:SL5}, we have 
$w_1^{-1}(\alpha_2) = - \alpha_1 - \alpha_2$ and
$w^{-1}(\alpha_2) = - \alpha_1 - \alpha_2 -\alpha_3$. 
Therefore, we have indeed $w_1^{-1}(\alpha_2) > w^{-1}(\alpha_2)$.  
Also, $(w'_1)^{-1}(\alpha_4) = -\alpha_4$ and 
$w^{-1}(\alpha_4) = -\alpha_2 -\alpha_3 -\alpha_4$, so that 
$(w'_1)^{-1}(\alpha_4) > w^{-1}(\alpha_4)$.
 
In the setting of Example \ref{ex:SL9}, we obtain
$w_1^{-1}(\alpha_3) = -\alpha_1 - \alpha_2 - \alpha_3 >
-\sum_{i=1}^6 \alpha_i = w^{-1}(\alpha_3)$ 
and 
$(w'_1)^{-1}(\alpha_8) = -\alpha_8 > - \sum_{i=4}^8 \alpha_i = w^{-1}(\alpha_8)$. 
\end{example}

We begin the proof of Proposition \ref{prop:root} with a preliminary result:

\begin{lemma}\label{lem:supp}
We have $\Supp(w_1) = \Supp(w_1^{-1}(\beta_1))$.
\end{lemma}

\begin{proof}
Recall that $v_i = s_{\beta_i} \cdots s_{\beta_1}$. Let $l = l_2$,  
then $v^{-1}_l=w_1$. Thus, it suffices to show that
$\Supp(v_i) = \Supp(v_i(\beta_1))$ for all $1\leq i \leq l$.  
As the inclusion $\Supp(v_i(\beta_1)) \subset \Supp(v_i)$
is obvious, it suffices in turn to prove the opposite inclusion. 

We argue by induction on $i$, $1\leq i \leq l$. If $i = 1$, then 
$\Supp(v_1) = \{ \beta_1 \} = \Supp(v_1(\beta_1))$.  
Let $1 \leq i \leq l-1$. 
By induction, we may assume that  
$\Supp(v_j) = \Supp(v_j(\beta_1))$ for all $1\leq j \leq i$. 
If $s_{\beta_{i+1}} \leq v_i$, we are done by Proposition \ref{prop:leq}. 
Otherwise, 
$v_{i+1}(\beta_1) = v_i(\beta_1) - 
\langle v_i(\beta_1) , \beta_{i+1} \rangle \beta_{i+1}$. 
Since $R^+(w_1^{-1})\cap S = \{ \beta_1 \}$, and since 
$s_{\beta_{i+1}} \not \leq v_i$, we have  
$\langle v_i(\beta_1) , \beta_{i+1} \rangle \geq 0$. 
Further, if  $\langle v_i(\beta_1) , \beta_{i+1} \rangle = 0$, 
then we have 
$s_{\beta_{i+1}} v_i = v_i s_{\beta_{i+1}}$. Hence, 
$v_{i+1}(\beta_{i+1})$ is a negative root. Therefore, we have  
$\beta_{i+1} \in R^+(w_1^{-1})  \cap S = \{ \beta_1 \}$, forcing 
$\beta_{i+1} = \beta_1$. This is a contradiction. Hence, we have 
$\langle v_i(\beta_1) , \beta_{i+1} \rangle \geq 1$.
Thus, we obtain $\Supp(v_{i+1}) =\Supp(v_{i+1}(\beta_1))$.  
\end{proof}

We may now prove Proposition \ref{prop:root}.
Since $R^+(w) \cap S = \{ \alpha \}$, there is an integer $2 \leq t \leq m$ 
such that $w_1$ does not commute with $w_t$. Let $2 \leq t_0 \leq m$ be
the least integer such that $w_1$ does not commute with $w_{t_0}$.  
  
Let $s = \ell_{t_0} + 1$ and $e = \ell_{t_0+1}$.  Then we have  
$w_{t_0} = s_{\beta_s} s_{\beta_{s+1}} \cdots s_{\beta_e}$.
By \cite[Definition 2.3]{Pe07} and \cite[Definition 4.4(i)]{Pe07}, it 
follows that for any $i<s$,  $\langle \beta_i, \beta_s \rangle \neq 0$ 
implies $i\preceq s$. This is a contradiction to $s$ being a peak.
Therefore, we have $\langle \beta_i , \beta_s \rangle =0$ 
for all $i<s$. So $s_{\beta_s}$ commutes with $w_i$ for all 
$1 \leq i \leq t_0-1$. In particular, $s_{\beta_s}$ commutes with $w_1$.

Let $s+1 \leq k \leq e$ be the least integer such that $s_{\beta_k}$ 
does not commute with $w_1$. Since $s_{\beta_j}$ commutes with 
$w_1$ for all $s \leq j \leq k-1$ and since any two reduced decompositions 
of $w$ differ only by commuting relations (see \cite[Prop.~2.1]{Stembridge}),  
$s_{\beta_j}$ commutes with $s_{\beta_f}$ for all $1 \leq f \leq l$ 
and for all $s \leq j \leq k-1$.  In particular, we have 
$s_{\beta_k} \not\leq w_1$.  Since $w_i$ commutes with $w_1$ for all 
$2 \leq i \leq t_0 - 1$, and $s_{\beta_j}$ commutes with $v_l = w_1^{-1}$ for all 
$s \leq j \leq k-1$, we have $v_k(\beta_1) = s_{\beta_k}(w_1^{-1}(\beta_1))$. 

On the other hand, we have   
\[ v_k(\beta_1) = s_{\beta_k}(w_1^{-1}(\beta_1)) 
= w_1^{-1}(\beta_1) - \langle w_1{-1}(\beta_1), \beta_k \rangle \beta_k. \] 
By Lemma \ref{lem:supp},  we have 
$\Supp(w_1) = \Supp(w_1^{-1}(\beta_1))$. Since $s_{\beta_k}$ 
does not commute with $w_1$, and $s_{\beta_k} \not\leq w_1$,  
we have $\langle w_1^{-1}(\beta_1), \beta_k \rangle \geq 1$. Thus, 
we have $v_k(\beta_1)  < w_1^{-1}(\beta_1)$. By Proposition \ref{prop:leq}
again, we have $w^{-1}(\beta_1) \leq v_k(\beta_1)$. So, we are done.

\subsection{Equality of Weyl groups}
\label{subsec:weyl}

We still keep the notation of Subsection \ref{subsec:seq}, and 
recall from Lemma \ref{lem:us} that 
$R^+(w_1^{-1}) \cap S = \{ \beta_1 \}$. 
Also, note that we have 
\[ \Peaks(Q_w)=\{1, l_2 + 1, l_3 + 1, \ldots, l_m + 1\}. \]
By Construction 1, we have $\Peaks(Q_{w})\cap Q_w(\{l_j + 1\})=\{ l_j + 1 \}$. 
On the other hand, $\beta(\Peaks(Q_w)) = R^+(w^{-1}) \cap S$. 
Thus, we have $R^+(w_j^{-1}) \cap S= \{ \beta_{l_j + 1} \}$.
Let $w' := w_2 w_3 \cdots w_m$. As
$\Peaks(Q_{w'})=\{l_2 + 1, l_3 + 1, \ldots, l_m + 1\}$, we have 
$R^+((w')^{-1})\cap S=\{ \beta_{l_2 + 1}, \beta_{l_3 + 1}, \ldots , \beta_{l_m + 1} \}$.

Let $T_{w_1}$ be the neutral component of $T \cap G_{w_1}$.
Note that $T_{w_1}$ is a maximal torus of $G_{w_1}$. Further, 
we have an isomorphism of Weyl groups 
$W(G_{w_1},T_{w_1}) \simeq W(L_{\Supp(w_1)}, T)$.
This identifies $W(G_{w_1},T_{w_1})$ with a subgroup of $W$.

\begin{lemma}\label{lem:length}
For any $u\in W(G_{w_{1}}, T_{w_{1}})$, we have 
$\ell(uw') = \ell(u )+ \ell(w')$.
\end{lemma}

\begin{proof} 
Since $\ell(v) = \vert R^+(v)  \vert$ for all $v \in W$, 
it suffices to show that $R^+((w')^{-1})\cap R^+(u)=\emptyset$.

Let $\gamma \in R^+((w')^{-1})$. Since 
$R^+((w')^{-1}) \cap S 
= \{ \beta_{l_2 + 1}, \beta_{l_3 + 1}, \ldots , \beta_{l_m + 1} \}$, 
there exists an integer $2 \leq j \leq m$ such that 
$\beta_{l_j + 1} \leq \gamma$. 
By \cite[Def.~2.3]{Pe07} and \cite[Def.~4.4 (i)]{Pe07},  
it follows that if there is an integer $1\leq f \leq l_2$
such that $\langle \beta_f , \beta_{l_j+1} \rangle \neq 0$, then we have
$f \preceq  1+l_j$. This is a contradiction to $1+l_j$ being a peak of $Q_w$.
Therefore, $s_{\beta_{1+l_j}}$ commutes with $s_{\beta_f}$ for all
$1 \leq f \leq l_2$.
In particular, we have $\beta_{l_j + 1} \notin \Supp(w_1)$, and 
$s_{\mu}(\beta_{l_j + 1})=\beta_{l_j + 1}$
for all $\mu \in \Supp(w_1)$.  Further, since 
$u \in  W(G_{w_1}, T_{w_1})$, we have 
$\Supp(u) \subset \Supp(w_1)$. Therefore, we have 
$u(\beta_{l_j + 1})=\beta_{l_j + 1}$. In particular, the coefficient of 
$\beta_{l_j + 1}$ in the expression of $u(\gamma)$ 
is equal to the coefficient of $\beta_{l_j + 1}$ in the expression of 
$\gamma$, and it is positive since $\beta_{l_j + 1} \leq \gamma$. 

Thus, we have $\gamma \notin R^+(u)$ as desired.
\end{proof}

\begin{example} 
With the notation of Example \ref{ex:SL5}, the standard ordering 
of $\Peaks(Q_w)$ has $w_1=s_2s_1$ and $w'=w_2=s_4s_3s_2$. Note that 
$R^+((w')^{-1})\subset \{ \beta\in R^+ ~\vert~ \alpha_4 \leq \beta \}$ 
and $s_4 \nleq w_1$. Therefore, for any $u\in W(G_{w_1}, T_{w_1})$, 
we have indeed $\ell(uw') = \ell(u) + \ell(w')$. 

For the reverse ordering of $\Peaks(Q_w)$, 
we have $w'_1 = s_4$, and $w' =w'_2 = s_2s_1s_3s_2$. Clearly, for any 
$u\in W(G_{w'_1}, T_{w'_1})$, we have again 
$\ell(uw') = \ell(u) + \ell(w')$. 
\end{example}

\begin{example}  
In the setting of Example \ref{ex:SL9}, for the standard ordering 
$\alpha_3 \preceq \alpha_5 \preceq \alpha_8$ of $\Peaks(Q_w)$, we have 
$w_1 = s_3s_2s_1$ and $w' = s_5s_4s_3s_2s_6s_5s_4s_3s_8s_7s_6s_5s_4$. 
As a consequence,
$R^+((w')^{-1}) \subset \{\beta ~\vert~ \alpha_5 \leq \beta \} \cup \{\alpha_8\}$. 
Therefore, for any $u\in W(G_{w_1}, T_{w_1}),$ we obtain 
$\ell(uw')=\ell(u)+\ell(w')$ as asserted.

For the ordering $\alpha_8 \preceq' \alpha_3 \preceq' \alpha_5$, 
we have $w'_1 = s_8$ and $w' = s_3s_2s_1s_5s_4s_3s_2s_6s_5s_4s_3s_7s_6s_5s_4$.
It readily follows that $\ell(uw') =\ell(u) + \ell(w')$
for any $u \in W(G_{w_1'}, T_{w_1'})$.  
\end{example}

\begin{proposition}\label{prop:weyl}
We have $W(G_{w_1,\hat{x}}, T_{w_1}) = W(G_{w_1, w_1 x_1}, T_{w_1})$.
\end{proposition}

\begin{proof}
Since $\hf : \hX(\hw) \to X(w_1)$ is $G_{w_1}$-equivariant
and $\hf(\hx) = w_1 x_1$,  we have  
\[ W(G_{w_1,  \hx}, T_{w_1}) \subset W(G_{w_1, w_1 x_1}, T_{w_1}). \]
Also, since $\hpi: \hX(\hw) \to X(w)$ is a local isomorphism at $\hx$
and $\hpi(\hx) = wx$, we have 
\[ W(G_{w_1,\hx}, T_{w_1}) = W(G_{w_1, wx}, T_{w_1}). \] 
Therefore, it suffices to prove that 
\[ W(G_{w_1, w_1 x_1}, T_{w_1}) \subset W(G_{w_1, wx}, T_{w_1}). \] 
Let $v \in W(G_{w_1, w_1 x_1}, T_{w_1})$. Then   
$v w_1 x_1 = w_1 x_1$ and hence there exists 
$\tau \in W(P^{w_1}, T)$ such that $v w_1=w_1 \tau$. Thus, we have 
$w_1\leq v w_1$ in $W(G_{w_1}, T_{w_1})$.  Note that both $v$ and $w_1$ 
are in $W(G_{w_1}, T_{w_1})$.  Therefore, by Lemma \ref{lem:length}, we have 
$\ell(v w_1 w') = \ell(v w_1) + \ell(w')$. So, $w = w_1 w'  \leq v w_1 w' = v w$. 
Since $w \in W^I$, it follows that $w \leq v'$ in $W^I$, where $v'$ denotes 
the minimal representative of $vw$ in $W^I$. 

On the other hand, we have $G_{w_1}\subset G_w \cap P_w$. Therefore, 
$v \in W(G_w\cap P_w, T_w)$.
Thus, we have $v X(w) = X(w)$. Therefore, $vw$ is in a coset 
$u W_I$ with $u\in W^I$ such that $u \leq w$. In particular,  
we have $v' \leq w$ in  $W^I$. Thus, we obtain $v'=w$. 
Therefore, we have $v w x=w x$ as desired.
\end{proof}

\bibliographystyle{amsalpha}

\begin{thebibliography}{A}


\bibitem[At58]{Atiyah}
M.~F.~Atiyah,
\textit{On analytic surfaces with double points},
Proc.~Royal Society London, Ser.~A, \textbf{247} 
(1958), 237--244.


\bibitem[Bo91]{Borel}
A.~Borel, 
\textit{Linear algebraic groups. Second enlarged edition},
Grad.~Texts in Math. \textbf{126}, Springer, 1991.


\bibitem[Bo08]{Bourbaki}
N.~Bourbaki,
\textit{Lie groups and Lie algebras. Chapters 4--6},
Springer, 2008.


\bibitem[BF15]{BF}
M.~Brion, B.~Fu,
\textit{Minimal rational curves on wonderful group compactifications},
J.~\'Ec. polytech. Math. \textbf{2} (2015), 153--170.


\bibitem[BK19]{BK}
M.~Brion, S.~Senthamarai Kannan,
\textit{Some combinatorial aspects of generalised Bott-Samelson varieties},
in preparation.


\bibitem[BP99]{BP}
M.~Brion, P.~Polo,
\textit{Generic singularities of certain Schubert varieties},
Math.~Z. \textbf{231} (1999), no.~2, 301--324.


\bibitem[BSU13]{BSU}
M.~Brion, P.~Samuel, V.~Uma,
\textit{Lectures on the structure of algebraic groups
and geometric applications}, Hindustan Book Agency, New Dehli, 2013; 
available online at \\
{\tt https://www-fourier.univ-grenoble-alpes.fr/$\;\widetilde{}\;$mbrion/chennai.pdf}.


\bibitem[CC98]{CC}
A.~Cohen, B.~Cooperstein,
\emph{Line incidence systems from projective varieties},
Proc. Amer. Math. Soc. \textbf{126} (1998), no.~7, 2095--2102.


\bibitem[CKP15]{CKP}
B.~Narasimha Chary, S.~Senthamarai Kannan, A.~J.~Parameswaran,
\textit{Automorphism group of a Bott-Samelson-Demazure-Hansen variety},
Transform. Groups \textbf{20} (2015), no.~3, 665--698.


\bibitem[De01]{Debarre}
O.~Debarre,
\textit{Higher-dimensional algebraic geometry},
Universitext, Springer, 2001.


\bibitem[DM91]{DM}
F.~Digne, J.~Michel,
\textit{Representations of finite groups of Lie type},
Cambridge Univ. Press, 1991.


\bibitem[Fu98]{Fulton}
W.~Fulton,
\textit{Intersection theory},
Ergebnisse der Math. \textbf{2}, Springer, 1998.


\bibitem[FW04]{FW}
W.~Fulton, C.~Woodward,
\textit{On the quantum product of Schubert classes},
J. Algebraic Geom. \textbf{13} (2004), no.~4, 641--661.


\bibitem[HL93]{HL}
W.~Haboush, N.~Lauritzen,
\textit{Varieties of unseparated flags}, 
in: Linear algebraic groups and their representations, 
Contemp. Math. \textbf{153}, 35--57, Amer. Math. Soc., 1993.


\bibitem[HM13]{HoM}
J.~Hong, N.~Mok,
\textit{Characterization of smooth Schubert varieties
in rational homogeneous manifolds of Picard number $1$},
J. Algebraic Geometry \textbf{22} (2013), no.~2, 333--362.


\bibitem[Ho15]{Hong}
J.~Hong,
\textit{Classification of smooth Schubert varieties 
in the symplectic Grassmannians},
J. Korean Math. Soc. \textbf{52} (2015), no.~5, 1109--1122.


\bibitem[HK19]{HoKw}
J.~Hong, M.~Kwon,
\textit{Rigidity of smooth Schubert varieties in a rational
homogeneous manifold associated to a short root},
arXiv:1907.09694.


\bibitem[Hu72]{Humphreys}
J.~E.~Humphreys,
\textit{Introduction to Lie algebras and representation theory},
Grad. Texts in Math. \textbf{9}, Springer, 1972.




\bibitem[HM02]{HwM}
J.-M.~Hwang, N.~Mok,
\textit{Deformation rigidity of the rational homogeneous space
associated to a long simple root},
Ann. Scient. \'Ec. Norm. Sup. \textbf{35} (2002), no.~2, 173--184.


\bibitem[Hw14]{Hwang14}
J.-M.~Hwang,
\textit{Mori geometry meets Cartan geometry: Varieties of minimal
rational tangents}, 
in: Proceedings of the International Conference of Mathematicians,
Seoul 2014, Vol.~I, 369--394, Kyung Moon SA, Seoul, 2014.


\bibitem[Ja03]{Jantzen}
J.~C.~Jantzen,
\textit{Representations of algebraic groups. Second edition},
Math. Surveys Monogr. \textbf{107}, Amer. Math. Soc., Providence,
2003.


\bibitem[Ke02]{Kebekus}
S.~Kebekus,
\textit{Families of singular rational curves},
J. Algebraic Geom. \textbf{11} (2002), no.~2, 245--256.


\bibitem[KR17]{KR}
M.~Kerr, C.~Robles,
\textit{Classification of smooth horizontal Schubert varieties},
European J. Math. \textbf{3} (2017), no.~2, 289--310.


\bibitem[Ko99]{Kollar}
J.~Koll\'ar,
\textit{Rational curves on algebraic varieties},
Ergebnisse der Math. \textbf{32}, Springer, Berlin, 1999.


\bibitem[LM03]{LM}
J.~Landsberg, L.~Manivel,
\textit{On the projective geometry of rational homogeneous varieties},
Comment. Math. Helv. \textbf{78} (2003), no.~1, 65--100.


\bibitem[LT04]{LT}
N.~Lauritzen, J.~F.~Thomsen,
\textit{Line bundles on Bott-Samelson varieties},
J. Algebraic Geom. \textbf{13} (2004), no.~3, 461--473.


\bibitem[MFK93]{MFK}
D.~Mumford, J.~Fogarty, F.~Kirwan,
\textit{Geometric invariant theory},
Ergeb. Math. Grenzgeb. \textbf{34}, Springer, 1993.


\bibitem[Pe05]{Pe05}
N.~Perrin,
\textit{Rational curves on minuscule Schubert varieties},
J. Algebra \textbf{294} (2005), no.~2, 431--462.


\bibitem[Pe07]{Pe07}
N.~Perrin,
\textit{Small resolutions of minuscule Schubert varieties},
Compositio Math. \textbf{143} (2007), no.~5, 1255--1312.


\bibitem[Pe09]{Pe09}
N.~Perrin,
\textit{Gorenstein locus of minuscule Schubert varieties},
Advances in Math. \textbf{220} (2009), no.~2, 505--522.


\bibitem[SV94]{SV94}
P.~Sankaran, P.~Vanchinathan,
\textit{Small resolutions of Schubert varieties 
in symplectic and orthogonal Grassmannians},
Publ. RIMS, Kyoto Univ. \textbf{30} (1994), no.~3, 443--458.


\bibitem[SV95]{SV95}
P.~Sankaran, P.~Vanchinathan,
\textit{Small resolutions of Schubert varieties 
and Kazhdan-Lusztig polynomials},
Publ. RIMS, Kyoto Univ. \textbf{31} (1995), no.~3, 465--480. 


\bibitem[St02]{Strickland}
E.~Strickland,
\textit{Lines in $G/P$},
Math. Z. \textbf{242} (2002), no.~2, 227--240.


\bibitem[St97]{Stembridge} J.~R.~Stembridge, 
{\it Minuscule elements of Weyl groups},
J. Algebra {\bf 235} (2001), no.~2, 722--745.


\bibitem[Ze83]{Zelevinsky}
A.~Zelevinsky,
\textit{Small resolutions of singularities of Schubert varieties},
Funct. Anal. Appl. \textbf{17} (1983), no.~2, 75--77.

\end{thebibliography}

\end{document}